\documentclass[10pt]{amsart}

\usepackage{amssymb}
\usepackage{amsmath}
\usepackage{amsfonts}
\usepackage{multirow}
\usepackage{stackrel}

\usepackage[all,cmtip,line]{xy}

\def\AA{\mathbb{A}}

\def\PP{\mathbb{P}}

\def\RR{\mathbb{R}}

\def\ZZ{\mathbb{Z}}

\newcommand{\calC}{{\mathcal{C}}}

\newcommand{\lra}{\longrightarrow}

\newcommand{\To}{\longrightarrow}
\newcommand{\isoto}{\stackrel{\sim}{\To}}

\newcommand{\mat}[4]{\left( \begin{array}{cc} {#1} & {#2} \\ {#3} & {#4}
\end{array} \right)}

\newcommand{\smat}[4]{{\mbox{\scriptsize $\mat{{#1}}{{#2}}{{#3}}{{#4}}$}}}
\newlength{\ownl}

\newcommand{\Ann}{{\operatorname{Ann}\,}}

\newcommand{\Aut}{{\operatorname{Aut}\,}}

\newcommand{\diag}{{\operatorname{diag}}}
\newcommand{\End}{{\operatorname{End}\,}}

\newcommand{\Hom}{{\operatorname{Hom}\,}}

\newcommand{\Spec}{{\operatorname{Spec}\,}}
\newcommand{\Spf}{{\operatorname{Spf}\,}}

\newcommand{\tr}{{\operatorname{tr}\,}}

\newcommand{\val}{{\operatorname{val}\,}}

\newcommand{\GL}{\operatorname{GL}}

\newcommand{\SL}{\operatorname{SL}}

\newcommand{\et}{{\operatorname{\acute{e}t}}}

\newcommand{\loc}{{\operatorname{loc}}}

\newcommand{\st}{{\operatorname{st}}}

\newcommand{\tor}{{\operatorname{tor}}}

\newcommand{\A}{{\mathbb{A}}}

\newcommand{\C}{{\mathbb{C}}}

\newcommand{\F}{{\mathbb{F}}}
\newcommand{\G}{{\mathbb{G}}}

\newcommand{\M}{{\mathbb{M}}}

\newcommand{\Q}{{\mathbb{Q}}}
\newcommand{\R}{{\mathbb{R}}}

\newcommand{\Z}{{\mathbb{Z}}}

\newcommand{\CA}{{\mathcal{A}}}

\newcommand{\CF}{{\mathcal{F}}}

\newcommand{\CH}{{\mathcal{H}}}
\newcommand{\CI}{{\mathcal{I}}}

\newcommand{\CK}{{\mathcal{K}}}
\newcommand{\CL}{{\mathcal{L}}}

\newcommand{\CO}{{\mathcal{O}}}

\newcommand{\CU}{{\mathcal{U}}}

\newcommand{\gP}{{\mathfrak{P}}}

\newcommand{\gY}{{\mathfrak{Y}}}

\newcommand{\ga}{{\mathfrak{a}}}

\newcommand{\gc}{{\mathfrak{c}}}
\newcommand{\gd}{{\mathfrak{d}}}

\newcommand{\gm}{{\mathfrak{m}}}
\newcommand{\gn}{{\mathfrak{n}}}

\newcommand{\gp}{{\mathfrak{p}}}
\newcommand{\gq}{{\mathfrak{q}}}
\newcommand{\gor}{{\mathfrak{r}}}

\renewcommand{\M}{\mathsf{M}}

\DeclareMathOperator{\ord}{ord}

\newcommand{\Fpbar}{\overline{\F}_p}

\newcommand{\tord}{\operatorname{tord}}
\newcommand{\minor}{\operatorname{minor}}
\newcommand{\spl}{\operatorname{spl}}
\newcommand{\SPEC}{\operatorname{\bf{Spec}}}
\newcommand{\HOM}{\mathcal{H}om}

\newcommand{\naive}{\operatorname{naive}}
\newcommand{\Tr}{\operatorname{Tr}}
\newcommand{\Nm}{\operatorname{Nm}}

\def\smallmat#1#2#3#4{\bigl(\begin{smallmatrix}{#1}&{#2}\\{#3}&{#4}\end{smallmatrix}\bigr)}

\usepackage{hyperref}

\newcommand{\e}{{\mathbf{e}}}
\newcommand{\f}{{\mathbf{f}}}
\newcommand{\uhp}{{\mathfrak{H}}}
\newcommand{\bq}{{\mathbf{q}}}

\newcommand{\utH}{{\underline{\widetilde{H}}}}
\newcommand{\unH}{{\underline{H}}}

\usepackage{mathrsfs}

\renewcommand{\L}{{\mathscr{L}}}

\theoremstyle{plain}
\newtheorem{theorem}{Theorem}[subsection]
\newtheorem{proposition}[theorem]{Proposition}
\newtheorem{corollary}[theorem]{Corollary}
\newtheorem{lemma}[theorem]{Lemma}

\theoremstyle{definition}
\newtheorem{remark}[theorem]{Remark}

\numberwithin{equation}{section}

\begin{document}
\title{Compactifications of Iwahori-level Hilbert modular varieties}
\author{Fred Diamond}

\address{Department of Mathematics, King's College London, Strand, 
London, WC2R 2LS, United Kingdom}

\email{fred.diamond@kcl.ac.uk}

%\author{Shu Sasaki}
%\email{s.sasaki.03@cantab.net}
%\address{School of Mathematical Sciences, Queen Mary University %of London, E1 4NS, UK}
%\subjclass[2010]{}
\begin{abstract} We study minimal and toroidal compactifications of $p$-integral models of Hilbert modular varieties.  We review the theory in the setting of Iwahori level at primes over $p$, and extend it to certain finer level structures.  We also prove extensions to compactifications of recent results on Iwahori-level Kodaira--Spencer isomorphisms and cohomological vanishing for degeneracy maps.  Finally we apply the theory to study $q$-expansions of Hilbert modular forms, especially the effect of Hecke operators at primes over $p$ over general base rings.
\end{abstract}

%\date{\today}
\maketitle

\section{Introduction}
In this paper we study compactifications of Hilbert modular varieties, augmenting and extending the existing theory in several ways we expect to be useful.

We start by recalling, in \S\S\ref{sec:U}--\ref{sec:U0}, some of the main results on toroidal and minimal compactifications of integral (at $p$) models of Hilbert modular varieties with (at most) Iwahori level structure at primes $\gp$ over $p$.  The models are defined using moduli problems introduced in~\cite{pappas} and~\cite{PR}, and
the results on compactifications can be obtained by adapting the methods of Rapoport (\cite{rap}; see also \cite{chai, Dim, theta}) or by specializing much more general results of Lan (especially \cite{KWL:spl}).  

The new results are presented in \S\S\ref{sec:U1}--\ref{sec:Tp}.  Firstly in \S\ref{sec:U1} we consider level $U_1(\gp)$-structure,
using models based on the moduli problem in~\cite{pappas}.
These are finite and flat over those with Iwahori (i.e., $U_0(\gp$)) level structure, and we construct and study compactifications which preserve this property.  These were already introduced in \cite{DS1} under the assumption that $p$ is unramified in the totally real field $F$, where they are used in establishing existence and properties of Galois representations associated to mod $p$ Hilbert modular eigenforms.  The results here will similarly be used in the sequel~\cite{DS2}, where the assumption on ramification is removed.

In \S\ref{sec:KS} we extend the main results of~\cite{KS} to toroidal compactifications.  Recall that in \cite{KS} we construct a Kodaira--Spencer isomorphism describing the dualizing sheaf of the integral Iwahori-level model, and prove relative cohomological vanishing results for degeneracy maps, the latter generalizing results in \cite{DKS}, where $p$ is assumed unramified in $F$.
Our main motivation for this is to extend the ``saving trace'' to compactifications.
Recall that the saving trace is introduced in \cite{KS} in order to conceptualize and generalize the construction of Hecke operators at primes over $p$.  Its extension here to toroidal compactifications is used in \S\ref{sec:Tp} to prove the operators $T_\gp$ have the desired effect on $q$-expansions.  (See \cite{DW, DMPhD, DS1} for similar formulas, based on different, less general, constructions of $T_\gp$.)  This also implies the commutativity of the operators $T_\gp$ (for varying $\gp$ over $p$), tying up a loose end from~\cite{KS}.

Finally  in \S\ref{sec:corr} we take the opportunity to list a few minor corrections to \cite{theta}.

\subsection*{Notation}
We adopt much of the set-up and notation from \cite{theta} and \cite{KS}.
In particular, we fix a prime $p$ and a totally real field $F \neq \Q$.  We let $\CO_F$ denote the ring of integers of $F$ and $\gd = \gd_{F/\Q}$ the different.  We write $\CO_{F,(p)}$ for the localization of $F$ at the prime $p$ of $\Z$, and $\CO_{F,p}$ for the $p$-adic completion of $\CO_F$.  We let $\Sigma_p$ denote the set of prime ideals of $\CO_F$ over $p$, so that $\CO_{F,(p)}$ is also the localization of $\CO_F$ at the complement of $\bigcup_{\gp \in \Sigma_p} \gp$ and $\CO_{F,p} = \prod_{\gp\in\Sigma_p} \CO_{F,\gp}$.  For each $\gp \in \Sigma_p$, we let $|\CO_F/\gp| = p^{f_\gp}$ and $e_\gp = v_\gp(p)$, so $[F_\gp:\Q_p] = e_\gp f_\gp$.

We also fix a finite extension $K$ of $\Q_p$, and let $\CO$ denote its ring of integers and $k$ its residue field.  We assume that $K$ is sufficiently large to contain the image of all embeddings of $F$ in the algebraic closure of $K$, and we let $\Theta$ denote the set of such embeddings.  For each $\gp \in \Sigma_p$ we let $\Theta_{\gp}$ denote the set of embeddings $F_{\gp} \to K$, and we identify $\Theta$ with $\coprod_{\gp\in \Sigma_p} \Theta_\gp$ via the canonical bijection.  Similarly for each $\gp \in \Sigma_p$ we may write $\Theta_{\gp} = \coprod_{\tau} \Theta_\tau$, where $\tau$ runs over the set of $f_\gp$ embeddings $\CO_F/\gp \to k$, and for each such $\tau$, we arbitrarily choose an ordering of 
$\Theta_\tau = \{\widetilde{\tau}_1,\ldots,\widetilde{\tau}_{e_\gp}\}$ of the $e_\gp$ elements of $\Theta_{\gp}$ restricting to $\widetilde{\tau}:W(\CO_F/\gp) \to \CO$.

For $\vec{n} \in \Z^\Theta$, let $\chi_{\vec{n}}:F^\times \to K^\times$ denote the corresponding character (defined by
$\chi_{\vec{n}}(\alpha) = \prod_{\theta\in \Theta} \theta(\alpha)^{n_\theta}$), and for any $\CO$-algebra $R$, we let $\chi_{\vec{n},R}:\CO_{F,(p)}^\times \to R^\times$ denote the character obtained from $\chi$ by composition.

We let $\A_{F,\f} = \widehat{\CO}_F \otimes \Q$ denote the finite adeles of $F$, and $\A^{(p)}_{F,\f} = \widehat{\CO}^{(p)}_F \otimes \Q$ the prime-to-$p$ finite adeles of $F$ (where $\widehat{\cdot}^{\,(p)}$ denotes prime-to-$p$ completion).  We let
$U$ be a sufficiently small\footnote{in the sense of \cite[\S2.2]{theta}} open compact subgroup of $\GL_2(\A_{F,\f})$ containing $\GL_2(\CO_{F,p})$, so that $U = U^p\GL_2(\CO_{F,p})$ for a (sufficiently small) open compact subgroup $U^p$ of
$\GL_2(\A^{(p)}_{F,\f})$.

\section{Level prime to $p$} \label{sec:U}

In this section, we recall the construction and properties of toroidal and minimal compactifications of $p$-integral models of Hilbert modular varieties of level prime to $p$.  Our main focus will be on the ``splitting'' models constructed by Pappas and Rapoport as in \cite{PR}, but we first consider the``naive'' models of Deligne and Pappas~\cite{DP}.  Since the ordinary loci of these models coincide, we may view the compactifications of the former as obtained from the latter.

\subsection{The Deligne--Pappas model} \label{ss:DP}

We let $\widetilde{Y}_-$ denote the (infinite disjoint union of Deligne--Pappas PEL) fine 
moduli scheme(s) of level $U$ (defined as in \cite[\S2.1]{theta}, but without filtrations), and $Y_- := \CO_{F,(p),+}^\times\backslash\widetilde{Y}_-$ 
the resulting model for the Hilbert modular  variety of level~$U$.
More precisely, for a locally Noetherian $\CO$-scheme $S$, 
$\widetilde{Y}_-(S)$ is identified with 
the set of isomorphism classes of data
$(A,\iota,\lambda,\eta)$, where:
\begin{itemize}
\item $s:A \to S$ is an abelian scheme of relative dimension $d$;
\item $\iota: \CO_F\to \End_S(A)$ is an embedding such that $(s_*\Omega^1_{A/S})_\gp$
is, locally on $S$, free of rank $e_\gp$ over $W(\CO_F/\gp)\otimes_{\ZZ_p} \CO_S $
for each $\gp \in S_p$;
\item $\lambda$ is an $\CO_F$-linear quasi-polarization of $A$ such that for each connected component
$S_i$ of $S$, $\lambda$ induces an isomorphism $\gc_i\gd \otimes_{\CO_F}  A_{S_i} \to A_{S_i}^\vee$
for some fractional ideal $\gc_i$ of $F$ prime to $p$;
\item $\eta$ is a level $U^p$ structure on $A$,
\end{itemize}
and the $\CO^\times_{F,(p),+}$-action is defined by 
$\nu\cdot(A,\iota,\lambda,\eta,\CF^\bullet) = (A,\iota,\nu\lambda,\eta,\CF^\bullet)$.
We thus have a universal abelian scheme
over $\widetilde{Y}_-$, and the determinant of its cotangent bundle along the zero section defines an
ample line bundle $\widetilde{\omega}$, descending to one on $Y_-$ which we denote by $\omega$.

\subsection{Cusps of level $U$} \label{ss:Ucusps}
We define the set of {\em cusps} of level $U$ as in \cite[\S7.1]{theta} (called there the cusps of $Y_U$ and denoted $Y_U^\infty$) 
to be
$$ C = C_U :=  B(F)_+ \backslash \GL_2(\A_{F,\f}) /  U = B(\CO_{F,(p)})_+ \backslash \GL_2(\A_{F,\f}^{(p)}) /  U^p,$$
where $B \subset \GL_2$ is the subgroup of upper-triangular matrices, and the subscript $+$ denotes those with totally positive determinant.
Similarly we define the set of {\em polarized cusps} of level $U$ to be
$$ \widetilde{C} = \widetilde{C}_U  :=  B_1(\CO_{F,(p)}) \backslash \GL_2(\A_{F,\f}^{(p)}) /  U^p,$$
where $B_1 = B \cap \SL_2$.
Thus $\widetilde{C}$  (resp.~$C$) is in bijection with the set of isomorphism classes of data
$\utH = (H, I, \lambda, [\eta])$ (resp.~$\unH = (H,I,[\lambda],[\eta])$), where
\begin{itemize}
\item $0 \to I \to H \to J \to 0$ is an exact sequence of projective $\CO_F$-modules with $I$ and $J := H/I$ invertible;
\item  $\lambda$ (resp.~$[\lambda]$) is an ($\CO_{F,(p),+}^\times$-orbit of) 
$\CO_{F,(p)}$-linear isomorphism(s) 
$$(IJ)_{(p)} \isoto \CO_{F,(p)};$$
\item $[\eta]$ is a $U^p$-orbit of $\widehat{\CO}_F^{(p)}$-linear isomorphisms 
$(\widehat{\CO}_{F}^{(p)})^2  \isoto  \widehat{H}^{(p)}$.
\end{itemize}
We write $[\utH] = [H,I,\lambda,[\eta]]$ (resp.~$[\unH] = [H,I,[\lambda,[\eta]]$) for the associated isomorphism class.

\subsection{Toroidal compactification of $\widetilde{Y}_-$} \label{ss:tY-tor}
For simplicity, we assume $U = U(N)$ for some integer $N \ge 3$ (prime to $p$) in the consideration of
toroidal compactifications, which are obtained as in \cite{DS1} or \cite{theta}
by applying the method of \cite[\S5]{rap} (see also \cite{chai}, \cite{Dim} or \cite{KWL:PhD}) to
the connected components of $\widetilde{Y}_-$.  More precisely, choosing a polyhedral cone 
decomposition (as in \cite[\S4]{rap}) of $(J^{-1}I\otimes\RR)_{\ge 0}$ for each cusp in $\widetilde{C}$
yields an open immersion $\widetilde{Y}_- \hookrightarrow \widetilde{Y}_-^{\tor}$,
where $\widetilde{Y}_-^{\tor}$ is an infinite disjoint union of flat projective schemes\footnote{a priori algebraic spaces, but in fact projective schemes by (the method of) \cite[\S7.3]{KWL:PhD}, or alternatively by the relation described in \S\ref{ss:KWL} with the compactifications defined in \cite{KWL:IMRN}} over $\CO$.

Let us assume furthermore that $\CO$ contains the $N^{\mathrm{th}}$ roots of unity. 
The connected components of the complement of $\widetilde{Y}_-$ in $\widetilde{Y}_-^{\tor}$ are then in canonical bijection with $\widetilde{C}$, and the completion of $\widetilde{Y}_-^{\tor}$ along the component corresponding to $[H,I,\lambda,[\eta]]$ is described explicitly as the quotient by a free action of $(\CO_F^\times \cap U)^2$ on a certain locally Noetherian formal scheme $\widehat{S}$, namely the completion of the complement of $\Spec(\CO[N^{-1}M])$ in the torus embedding defined by the chosen $(\CO_F^\times \cap U)^2$-invariant
cone decomposition of $\Hom(M,\R)_{\ge 0}$, where $M = \gd^{-1}I^{-1}J$.

We remark that the cone decompositions can be chosen so that the action of $\CO_{F,(p),+}^\times$ on $\widetilde{Y}_-$ extends to one on $\widetilde{Y}_-^{\tor}$, yielding a toroidal compactification $Y_-^{\tor}$ of $Y_-$ as the quotient.
One can furthermore remove the restriction that $U = U(N)$, but we will make no use of this here, our ultimate focus (except in \S\ref{sec:KS}) being on the minimal compactifications for more general $U$.

\subsection{Relation with Lan's definition} \label{ss:KWL}
We note that our $\widetilde{Y}_-$ is an infinite disjoint union of schemes of the form
$\M_{\CH^p,\CO}^{\naive}$ in the notation of \cite{KWL:spl}.  More precisely for each $\epsilon \in (\A_{F,\f}^{(p)})^\times/(\det U^p)(\widehat{\Z}^{(p)})^\times$, let $\widetilde{Y}_-^{\epsilon}$ denote the open and closed subscheme
of $\widetilde{Y}_-$ over which the diagram
$$\xymatrix{ \gd^{-1}\widehat{\CO}_D^{(p)} \times \widehat{\CO}_D^{(p)}
\ar[r]^{\psi_F}  \ar[d]_{(\eta_i,\eta_i)}  &
\A_{F,\f}^{(p)}
\ar[d]^{\epsilon_i}  \\
T^{(p)}(A_{\overline{s}_i})  \times (\gd\otimes_{\CO_F}T^{(p)}(A_{\overline{s}_i})) 
\ar[d]_{(1,\lambda)}  &
\A_{F,\f}^{(p)}(1)
\ar[d]^{\Tr_{F/\Q}} \\
T^{(p)}(A_{\overline{s}_i})  \times (\Q \otimes T^{(p)}(A^\vee_{\overline{s}_i}) )
\ar[r]^-{\mathrm{Weil}} &
\A_{\f}^{(p)}(1)}$$
commutes (in the notation of \cite[\S2.2]{theta}, cf.~\cite[\S2.1.2]{DKS}).
Letting $[\epsilon]$ denote the fractional ideal of $F$ defined
by $\epsilon$ and
$L = \delta\gd^{-1} \oplus [\epsilon]\gd^{-1}$ for any
totally positive $\delta \in [\epsilon]^{-1}$ 
such that $\delta\CO_{F,p} = \gd \CO_{F,p}$,
the scheme $\widetilde{Y}_-^{\epsilon}$ is then isomorphic to
the one defined in \cite{KWL:spl} as $\M_{\CH^p,\CO}^{\naive}$ using the standard alternating pairing on $L$ (composed with the trace and twisted by any choice of $\Z \isoto \Z(1)$).  The isomorphism (and choice of $\CH^p$) are
given by modifying the level structure (resp.~quasi-polarization) of \cite{theta} by multiplication
 by $\diag(\delta^{-1},\epsilon^{-1})$ (resp.~$\delta$). 

By Corollary~2.4.8 of \cite{KWL:spl}, we also have an isomorphism 
$$\vec{\M}_{\CH,\CO} \stackrel{\sim}{\longrightarrow} \M_{\CH^p,\CO}^{\loc} = \M_{\CH^p,\CO}^{\naive}$$
(since $\M_{\CH^p}^{\naive}$ is flat over $\ZZ_p$ and normal, $\M_{\CH^p,\CO}^{\spl} \to \M_{\CH^p,\CO}^{\naive}$
is surjective and $\M_{\CH^p,K}^{\naive} = \M_{\CH,K}$, where $\CH = \CH^p\CU_p(\L)$ and $\L$ is the set of lattices in $F_p^2$ of the form $\ga \oplus \gd^{-1}\ga$ where $\ga$ is an invertible $\CO_{F,p}$-submodule of $F_p$).
Furthermore the condition in Theorem~6.1(6) of \cite{KWL:IMRN}
 is satisfied\footnote{To
 make the translation between our set-up and that of \cite{KWL:IMRN} at the polarized cusp $(H,I,\lambda,[\eta])$ 
of level $U = U(N)$, take $\underline{X}^{\ddag} = I^{-1}$, $\underline{Y}^{\ddag} = \gd^{-1}J$,
and
$\phi^\ddag: \underline{X}^\ddag \to \underline{Y}^\ddag$ to be $\delta$ times the homomorphism induced by $\lambda$, with
$\mathtt{Z}_\CH^\ddag$ and $[\alpha_\CH^{\natural,\ddag}]$ determined by the inclusion $I \subset H$ and 
the isomorphism $\gd L/N\gd L \isoto H/NH$ defined by 
$\eta\circ \diag(\delta^{-1},\epsilon^{-1})$.} by the connected components of $\widetilde{Y}_-^{\tor}$, yielding morphisms to the toroidal compactificatons $\vec{\M}^{\tor}_{\CH,\Sigma,\CO}$ satisfying the conclusion of \cite[Thm.~3.5]{rap}, so that the identification above extends to one between $\widetilde{Y}_-^{\tor}$ and an infinite
disjoint union of toroidal compactifications $\vec{\M}^{\tor}_{\CH,\Sigma,\CO}$ as in \cite{KWL:IMRN}.

\subsection{Minimal compactification of $\widetilde{Y}_-$} \label{ss:tY-min}
We continue to assume for the moment that $U$ is of the form $U(N)$.  The universal abelian scheme over $\widetilde{Y}_-$ extends to a semi-abelian scheme
over $\widetilde{Y}_-^\tor$, yielding also an extension of the line bundle $\widetilde{\omega}$.
Moreover the line bundle $\omega_{\vec{\M}^{\tor}_{\CH,\Sigma},J}$ on $\vec{\M}^{\tor}_{\CH,\Sigma,\CO}$
(with $J$ a singleton) is identified with the pull-back of our $\widetilde{\omega}^{\otimes a}$ for some
integer $a > 0$ by \cite[Thm.~6.1(2)]{KWL:IMRN}.  The scheme there denoted $\vec{\M}_{\CH,\CO}^{\min}$
therefore coincides with the projective scheme associated to the global sections of the symmetric algebra on
$\widetilde{\omega}$ over the corresponding connected components, and taking their
disjoint union yields the minimal compactification $\widetilde{Y}_- \hookrightarrow \widetilde{Y}_-^{\min}$.

The scheme $\widetilde{Y}_-^{\min}$ is normal, independent of the choice of cone decompositions in the definition of
$\widetilde{Y}_-^{\tor}$, and its connected components are flat and projective
over $\CO$.   We thus have a commutative diagram of morphisms
$$\xymatrix{&\widetilde{Y}_-^{\tor}\ar[dd]  \\
\widetilde{Y}_- \ar[rd]\ar[ru] & \\
& \widetilde{Y}_-^{\min} }$$
over $\CO$, where the diagonal morphisms are open immersions, and the vertical morphism is projective.
Furthermore the reduced complement of $\widetilde{Y}_-$ in 
$\widetilde{Y}_-^{\min}$ is \'etale over $\CO$ with geometric connected
components indexed by $\widetilde{C}$, and the
Koecher Principle (in the form of \cite[Thm.~8.7]{KWL:IMRN}) applies to give an explicit description of the
completion along this complement, as in \cite[(23)]{theta}.

\subsection{Minimal compactification of ${Y}_-$} \label{ss:Y-min}
The action of $\CO_{F,(p),+}^\times/(U\cap \CO_F^\times)^2$ on $\widetilde{Y}_-$
extends uniquely to an action on $\widetilde{Y}_-^{\min}$, and we define $Y_-^{\min}$ to be the quotient.
Furthermore for any sufficiently small level $U$ prime to $p$, we may choose $N \ge 3$ (prime to $p$) 
so that $U' := U(N) \subset U$ and define $Y_-^{\min}$ to be the quotient of $Y_-^{\prime\min}$ by the (unique
extension of the) natural action of $U$, where $Y_-^{\prime\min}$ is defined as above.  The resulting scheme
$Y_-^{\min}$ is then normal, flat and projective over $\CO$, and independent of the choice of $N$.
Furthermore the reduced complement of the image of the open immersion $Y_- \hookrightarrow Y_-^{\min}$
is \'etale over $\CO$ with geometric connected components indexed by $C$.  We assume $\CO$ is sufficiently large that the components are defined over $\CO$, and refer to them also as {\em cusps}.
The completion of $Y_-^{\min}$ at the cusp corresponding to $\underline{H}$ is decribed by (the displayed equation preceding) \cite[Prop.~7.2.1]{theta}.

\subsection{The Pappas--Rapoport model} \label{ss:PR}
We let $\widetilde{Y} = \widetilde{Y}_U$ denote the scheme defined in \cite[\S2.1]{theta}, obtained by equipping the universal object over $\widetilde{Y}_-$
with Pappas--Rapoport filtrations, 
and let $Y = \CO_{F,(p),+}^\times\backslash\widetilde{Y}$ denote the resulting smooth model for the
Hilbert modular variety.  Thus if $S$
 is a locally Noetherian $\CO$-scheme, then $\widetilde{Y}(S)$ is the set of isomorphism classes of data $(A,\iota,\lambda,\eta,\{\CF_\tau^\bullet\})$, where $(A,\iota,\lambda,\eta)$ defines an element of $\widetilde{Y}_-(S)$ and for each $\gp \in S_p$ and $\tau:\CO_F/\gp \to k$, $\CF_\tau^\bullet$ is an increasing filtration of 
$\CO_{F,\gp} \otimes_{W(\CO_F/\gp),\widetilde{\tau}} \CO_S$-modules 
$$0 = \CF_\tau^{(0)} \subset \CF_\tau^{(1)} \subset \cdots 
   \subset \CF_\tau^{(e_\gp - 1)} \subset \CF_\tau^{(e_\gp)} =  
   (s_*\Omega_{A/S}^1)_{\widetilde{\tau}}$$
such that for $j=1,\ldots,e_{\gp}$, the quotient
$${\CL}_{\tau,j}  :=  \CF_\tau^{(j)}/\CF_\tau^{(j-1)}$$
is a line bundle on $S$ on which $\CO_F$ acts via $\widetilde{\tau}_j$.

\subsection{Toroidal compactification of $\widetilde{Y}$} \label{ss:tYtor}
Again assuming $U$ is of the form $U(N)$ for some $N \ge 3$, 
the toroidal compactification $\widetilde{Y} \hookrightarrow \widetilde{Y}^{\tor}$
is defined similarly to that of $\widetilde{Y}_-$ (see \S\ref{ss:tY-tor}).  Thus $\widetilde{Y}^\tor$ is an infinite disjoint union of
projective schemes\footnote{The scheme $\widetilde{Y}^\tor$ is even smooth over
$\CO$ for suitable choice of cone decompositions.} over $\CO$, and the forgetful morphism
$\widetilde{Y} \to \widetilde{Y}_- $ extends to a projective morphism
$\widetilde{Y}^{\tor} \to \widetilde{Y}_-^\tor$.  Furthermore the
universal abelian scheme $A$ over $\widetilde{Y}$ extends to a semi-abelian scheme
$A^{\tor}$ over $\widetilde{Y}^{\tor}$.

Let $\widetilde{Y}^{\tord}$ (resp.~$\widetilde{Y}_-^{\tord}$) denote the ordinary locus in
$\widetilde{Y}^{\tor}$ (resp.~$\widetilde{Y}_-^{\tor}$), defined as the complement of the vanishing
locus of the Hasse invariant, viewed as a global section of the pull-back of $\widetilde{\omega}^{\otimes(p-1)}$
to the special fibre $\widetilde{Y}_k^{\tor}$ (resp.~$\widetilde{Y}_{-,k}^{\tor}$).  The existence and uniqueness
of Pappas--Rapoport filtrations over $\widetilde{Y}^{\tord}$ implies that the 
morphism $\widetilde{Y}^{\tor} \to \widetilde{Y}_-^{\tor}$ restricts to an isomorphism
$\widetilde{Y}^{\tord} \stackrel{\sim}{\lra} \widetilde{Y}_-^{\tord}$.  
We note also that $\widetilde{Y}^{\tor} = \widetilde{Y} \cup \widetilde{Y}^{\tord}$
and $\widetilde{Y}_-^{\tor} = \widetilde{Y}_- \cup \widetilde{Y}_-^{\tord}$,
and defining $\widetilde{Y}^{\ord} = \widetilde{Y} \cap \widetilde{Y}^{\tord}$ and
$\widetilde{Y}_-^{\ord} = \widetilde{Y}_- \cap \widetilde{Y}_-^{\tord}$,
the above isomorphism restricts to $\widetilde{Y}^{\ord} \stackrel{\sim}{\lra} \widetilde{Y}_-^{\ord}$.

Our $\widetilde{Y}$ is now an infinite disjoint union of the schemes
$\M_{\CH^p,\CO}^{\spl}$ of\footnote{Note that our filtrations
are increasing, whereas those in \cite{KWL:spl} are decreasing, and the condition in \cite[Def.~2.3.3(4)]{KWL:spl}
is automatic by \cite[Lemma~3.1.1]{theta}.} \cite{KWL:spl}, whose Corollary~2.4.10 gives an isomorphism 
$$\vec{\M}^{\spl}_{\CH,\CO} \stackrel{\sim}{\longrightarrow} \M_{\CH^p,\CO}^{\spl}.$$
(again with $\CH = \CH^p\CU_p(\L)$).  Just as for $\widetilde{Y}_-$, it follows that the isomorphism
extends to one between $\widetilde{Y}^{\tor}$ and an infinite disjoint of the schemes
$\vec{\M}^{\spl,\tor}_{\CH,\Sigma,\CO}$ defined in \cite{KWL:spl}, where the universal property
is now the one in \cite[Thm.~3.4.1(4)]{KWL:spl}.

\subsection{Minimal compactification of $\widetilde{Y}$} \label{ss:tYmin}
Letting $\widetilde{Y}^{\minor}$ denote the ordinary locus in $\widetilde{Y}_-^{\min}$,
defined as in \S\ref{ss:tYtor}, we have $\widetilde{Y}_-^{\min} = \widetilde{Y}_- \cup \widetilde{Y}^{\minor}$,
and we construct the  minimal compactification $\widetilde{Y}^{\min}$ by
gluing $\widetilde{Y}$ to $\widetilde{Y}^{\minor}$ along
$\widetilde{Y}^{\ord} \stackrel{\sim}{\lra} \widetilde{Y}_-^{\ord}$.
The scheme $\widetilde{Y}^{\min}$ is thus normal and independent of choice of cone decompositions,
its connected components are flat and projective over $\CO$, and we 
have a commutative diagram of morphisms
$$\xymatrix{&\widetilde{Y}^{\tor}\ar[dd]  \\
\widetilde{Y} \ar[rd]\ar[ru] & \\
& \widetilde{Y}^{\min} }$$
over the corresponding ones for $\widetilde{Y}_-$; 
again the diagonal morphisms are open immersions and the vertical morphism is projective.

We claim also that $\widetilde{Y}^{\min}$ 
is isomorphic to an infinite disjoint of the schemes denoted
$\vec{\M}^{\spl,\min}_{\CH,\CO}$ in \cite{KWL:spl}.  Indeed by the Koecher
Principle for the vertical morphisms in the diagram
$$\begin{array}{ccc} \vec{\M}^{\spl,\tor}_{\CH,\Sigma,\CO}  &\longrightarrow & \vec{\M}^{\tor}_{\CH,\Sigma,\CO} \\
\downarrow&&\downarrow \\ \vec{\M}^{\spl,\min}_{\CH,\CO} & \longrightarrow &  \vec{\M}^{\min}_{\CH,\CO} \end{array}$$
(\cite[Thm.~4.4.10]{KWL:spl}, \cite[Thm.~8.7]{KWL:IMRN}),
the fact that the top arrow is an isomorphism on the ordinary locus implies
that so is the bottom arrow.

For $\vec{k},\vec{m} \in \ZZ^\Theta$ and $R$ a Noetherian $\CO$-algebra,
let $\widetilde{\CA}_{\vec{k},\vec{m},R}$ be the associated line bundle on $\widetilde{Y}_{R}$ (see \cite[\S3.2]{theta}).
The extension of the (semi-)abelian scheme $A$ to $A^{\tor}$ over $\widetilde{Y}^{\tor}$ yields an extension
of $\widetilde{\CA}_{\vec{k},\vec{m},R}$ to $\widetilde{Y}_R$ which is formally canonical in the sense of \cite[Def.~8.5]{KWL:IMRN}
(see \cite[(20)]{theta}).  Therefore the Koecher Principle (\cite[Thm.~4.4.10]{KWL:spl}) applies to show that 
$\widetilde{j}_*\widetilde{\CA}_{\vec{k},\vec{m},R}$ is coherent, where $\widetilde{j}$ denotes the open immersion
$\widetilde{Y}_R \hookrightarrow \widetilde{Y}_R^{\min}$, and that its completion along the complement of $\widetilde{Y}_R$
is described as in \cite[(22)]{theta}.  

\subsection{Minimal compactification of $Y$} \label{ss:Ymin} For $U = U(N)$, we define $Y^{\min}$ to be the quotient of $\widetilde{Y}^{\min}$
by the (unique extension of the) action of $\CO_{F,(p),+}^\times/(U\cap \CO_F^\times)^2$.  More generally for any 
sufficiently small level $U$ prime to $p$, we let $Y^{\min}$ to be the quotient of $Y^{\prime\min}$ by the
(unique extension of the) action of $U$, where $Y^{\prime\min}$ is defined using $U' = U(N)$ for suitable $N$.  Just as for $Y_-$,
the resulting scheme $Y^{\min}$ is normal, flat and projective over $\CO$, and independent of the choice of $N$.
Furthermore the reduced complement of the image of the open immersion $Y \hookrightarrow Y^{\min}$, and
the completion along it, are the same as for $Y_- \hookrightarrow Y_-^{\min}$.  

Recall that if $\chi_{\vec{k} + 2\vec{m},R} = 1$ on $\CO_F^\times \cap U$, then the line bundle
$\widetilde{\CA}_{\vec{k},\vec{m},R}$ on $\widetilde{Y}_R$ descends to one on $Y_R$, which we denote
$\CA_{\vec{k},\vec{m},R}$.  Letting $i$ denote the open immersion $Y_R \hookrightarrow Y_R^{\min}$,
it follows from the analogous statements over $\widetilde{Y}_R^ {\prime\min}$
that $i_*{\CA}_{\vec{k},\vec{m},R}$ is coherent and that its completion along the complement of
the image of $i$ is given by \cite[Prop.~7.2.1]{theta}.

\section{Iwahori level at $p$} \label{sec:U0}

We now recall how the theory reviewed in \S\ref{sec:U} applies to
yield toroidal and minimal compactifications
of Hilbert modular varieties with Iwahori level at primes over $p$.

\subsection{The Iwahori-level model}\label{ss:Iw}
Let $\gP$ be a divisor of the radical of $p\CO_F$, and let
$$U_0(\gP) = \left\{\left.\,\left(\begin{array}{cc}a&b\\c&d\end{array}\right) \in U \,\right|\,\mbox{$c_\gp \in \gp\CO_{F,\gp}$ for all $\gp|\gP$}\,\right\}.$$
We let  $\widetilde{Y}_0(\gP)$ denote the corresponding fine moduli scheme parametrizing pairs of objects
$\underline{A}_1$, $\underline{A}_2$ of $\widetilde{Y}$, equipped with a $\gP$-isogeny $\psi:A_1 \to A_2$
respecting the additional structures (as defined in \cite[\S2.4]{KS}, thus depending on a choice of
$\varpi_{\gP}$ in the notation there).  Similarly we let $Y_0(\gP)$
denote its quotient by the action of $\CO_{F,(p),+}^\times$, so that $Y_0(\gP)$ is a model for
the Hilbert modular variety of level $U_0(\gP)$ (and is independent of $\varpi_{\gP}$).   We thus have a pair of forgetful morphisms
$$\widetilde{\pi}_1,\widetilde{\pi}_2:  \widetilde{Y}_0(\gP)  \longrightarrow \widetilde{Y}$$
inducing morphisms $\pi_1,\pi_2: Y_0(\gP) \to Y$ which on complex points correspond to the maps
$$\GL_2(F)_+\backslash  (\uhp^\Theta \times \GL_2(\A_{F,\f}) / U_0(\gP)) \longrightarrow 
\GL_2(F)_+\backslash  (\uhp^\Theta \times \GL_2(\A_{F,\f}) / U)$$
defined by the natural projection and multiplication by $\alpha_\gP := \prod_{\gp|\gP} \smat{1}{0}{0}{\varpi_\gp}_\gp$
for any choice of uniformizers $\varpi_\gp$ at $\gp$.

\subsection{Cusps of level $U_0(\gP)$}\label{ss:U0cusps}
We define the set of {\em  cusps} of level $U_0(\gP)$, denoted $C_0(\gP)$, exactly as we did for level $U$ (see \S\ref{ss:Ucusps}), but with $U$ now replaced
by $U_0(\gP)$.  We thus have the natural projections
$$\pi_1^{\infty},\pi_2^{\infty}:  C_0(\gP) 
= B(F)_+ \backslash \GL_2(\A_{F,\f}) /  U_0(\gP) \to C$$
defined by $B(F)_+ g U_0(\gp) \mapsto B(F)_+ g U$,  
$B(F)_+ g \alpha_\gP U$.
Furthermore we have a bijection between $C_0(\gP)$ and the set of isomorphism classes of triples
$(\underline{H}_1,\underline{H}_2,\alpha)$, where $\underline{H}_i = (H_i,I_i,[\lambda_i],[\eta_i])$
correspond to elements of $C$ for $i=1,2$, and $\alpha:H_1 \to H_2$ is an $\CO_F$-linear
homomorphism such that
\begin{itemize}
\item $\alpha(I_1) \subset I_2$;
\item $H_2/\alpha(H_1)$ is isomorphic to $\CO_F/\gP$;
\item $[\lambda_1] = [\varpi_\gp^{-1}\lambda_2 \circ \wedge^2\alpha]$;
\item $[\eta_2] = [\widehat{\alpha}^{(p)}\circ \eta_1]$.
\end{itemize}
Under this bijection the maps $\pi_i^{\infty}$ correspond to the
obvious forgetful maps, and we have a pair of bijections 
$$C_0(\gP) \longrightarrow C \times \{\,\gq \,|\, \gP \subset \gq \subset \CO_F\,\}$$
defined by $[\underline{H}_1,\underline{H}_2,\alpha] \mapsto ([\underline{H}_i], \gq_i )$
for $i=1,2$, where $\gq_1 = \Ann_{\CO_F}(J_2/\alpha(J_1)))$, $\gq_2 = \Ann_{\CO_F}(I_2/\alpha(I_1))= \gq_1^{-1}\gP$, and we write $[\unH_1,\unH_2,\alpha]$ for the isomorphism class of $(\unH_1,\unH_2,\alpha)$.

We also let $\widetilde{C}_0(\gP)$ denote the set of isomorphism
classes of triples $(\utH_1,\utH_2,\alpha)$, where now the
$\utH_i$ (resp.~$\alpha$) are as in the description of $\widetilde{C}$ (resp.~$C_0(\gP))$,
with the additional condition that $\lambda_2 \circ \wedge^2\alpha =\varpi_\gP\lambda_1$.  We then have a pair of bijections
$$\widetilde{C}_0(\gP) \longrightarrow \widetilde{C} \times \{\,\gq \,|\, \gP \subset \gq \subset \CO_F\,\}$$
defined in the same way as for $C_0(\gP)$, and we write $[\utH_1,\utH_2,\alpha]$ for the isomorphism class of $(\utH_1,\utH_2,\alpha)$.

\subsection{Toroidal compactification of $\widetilde{Y}_0(\gP)$} \label{ss:tY0tor}
Once again we assume that $U= U(N)$ for some $N \ge 3$, with the $N^{\mathrm{th}}$ roots of unity contained in $\CO$, and define the toroidal compactification
$\widetilde{Y}_0(\gP) \hookrightarrow \widetilde{Y}_0(\gP)^{\tor}$ as in \cite{rap}.  More precisely for
each cusp in $\widetilde{C}_0(\gP)$, we choose an admissible cone decomposition of 
$(J_2^{-1}I_1\otimes\RR)_{\ge 0}$ and construct the toroidal
compactification using the morphism of semi-abelian schemes defined by
\begin{equation}\label{eqn:Tate} (\gd^{-1}I_1 \otimes \G_m)/\widetilde{q}^{\gd^{-1}J_1} \longrightarrow (\gd^{-1}I_2 \otimes \G_m)/\widetilde{q}^{\gd^{-1}J_2},\end{equation}
over the resulting formal scheme.  The universal isogeny $\psi: A_1 \to A_2$ over $\widetilde{Y}_0(\gP)$ then extends to a
morphism $A_1^\tor \to A_2^\tor$ of semi-abelian schemes over $\widetilde{Y}_0(\gP)^\tor$ whose completion along the complement
of $\widetilde{Y}_0(\gP)$ is described by (\ref{eqn:Tate}).

Our $\widetilde{Y}_0(\gP)$ can again be identified with an infinite disjoint union of schemes of the form $\M_{\CH^p,\CO}^{\spl}$ considered in \cite{KWL:spl}, where $\L$ is now the set of lattices in $F_p^2$ of the form $\ga \oplus \gd^{-1}\gq\ga$ for an invertible $\CO_{F,p}$-submodule $\ga$ of $F_p$ and an ideal $\gq$ of $\CO_F$ containing $\gP$.  (See \cite[\S2.4]{KS}, especially (3) for the condition in \cite[Def.~2.3.3(4)]{KWL:spl}.)  Just as for $\widetilde{Y}$, it follows from
\cite[Cor.~2.4.10]{KWL:spl} that $\widetilde{Y}_0(\gP)$ is
an infinite disjoint union of schemes of the form
$\vec{\M}^{\spl}_{\CH,\CO}$, where now $\CH = \CH^p \CU_p(\L)$
for this choice of $\L$, and that the identification
extends to one between $\widetilde{Y}_0(\gP)^{\tor}$ and an infinite disjoint of the schemes
$\vec{\M}^{\spl,\tor}_{\CH,\Sigma,\CO}$ defined in \cite{KWL:spl}.  Furthermore if the admissible cone
decomposition for each cusp $[\utH_1,\utH_2,\alpha]$ in $\widetilde{C}_0(\gP)$
is chosen to refine those for the cusps $[\utH_1]$ and $[\utH_2]$ in $\widetilde{C}$,
then the same universal property implies that each $\widetilde{\pi}_i$ extends to a
morphism $\widetilde{\pi}_i^{\tor}:\widetilde{Y}_0(\gP)^\tor \to \widetilde{Y}^\tor$ under which the semi-abelian
scheme $A_i^\tor$ is the pull-back of the extension $A^\tor$ of the universal abelian
scheme $A$ over $\widetilde{Y}$.

We summarize the main results as follows:
\begin{theorem} \label{thm:tY0tor} There is a normal scheme 
$\widetilde{Y}_0(\gP)^{\tor}$ over $\CO$ and an open immersion 
$\widetilde{Y}_0(\gP) \hookrightarrow\widetilde{Y}_0(\gP)^{\tor}$ with the following properties:
\begin{enumerate}
\item $\widetilde{Y}_0(\gP)^{\tor}$ is flat and Cohen--Macaulay    over $\CO$ with projective connected components, and there is a canonical bijection $\widetilde{c} \longleftrightarrow \widetilde{Z}_{\widetilde{c}}$ between $\widetilde{C}_0(\gP)$ and the set of connected components of the reduced complement of $\widetilde{Y}_0(\gp)$ in $\widetilde{Y}_0(\gP)^\tor$; 
\item the completion of $\widetilde{Y}_0(\gP)^{\tor}$ along
$\widetilde{Z}_{\widetilde{c}}$ is isomorphic to $\widehat{S}_0(\gP)/(\CO_F^\times \cap U)^2$, where $\widehat{S}_0(\gP)$ is the  completion of the complement of $\Spec(\CO[N^{-1}M])$ in the torus embedding defined by the chosen cone decomposition of $ \Hom(M,\R)_{\ge 0}$, $\widetilde{c} = [\utH_1,\utH_2,\alpha]$ and $M = \gd^{-1}I_1^{-1}J_2$;
\item the cone decompositions may be chosen so that the degeneracy
maps $\widetilde{\pi}_1$ and $\widetilde{\pi}_2$ extend to morphisms $\widetilde{Y}_0(\gP)^{\tor} \to \widetilde{Y}^\tor$, and the universal isogeny $\psi:A_1 \to A_2$ over $\widetilde{Y}_0(\gP)$ extends to an isogeny of semi-abelian schemes over $\widetilde{Y}_0(\gP)^{\tor}$ whose completion along $\widetilde{Z}_{\widetilde{c}}$ has the form (\ref{eqn:Tate}).
\end{enumerate}
\end{theorem}

Just as for $Y_-$ (see \S\ref{ss:tY-tor}), one can obtain a toroidal compactification $Y_0(\gP)^\tor$ of $Y_0(\gP)$ as a quotient of $\widetilde{Y}_0(\gP)^\tor$.  This applies in particular to $Y = Y_0(\CO_F)$; furthermore one can relax the restriction that $U = U(N)$.

\subsection{Minimal compactification of $\widetilde{Y}_0(\gP)$} \label{ss:tY0min}
Let $\widetilde{Y}_0(\gP)^{\ord}$ (resp.~$\widetilde{Y}_0(\gP)^{\tord}$) denote the ordinary locus in 
$\widetilde{Y}_0(\gP)$ (resp.~$\widetilde{Y}_0(\gP)^{\tor}$), and define 
$$\widetilde{Y}_0(\gP)^{\minor} = \SPEC(f_*\CO_{\widetilde{Y}_0(\gP)^{\tord}}),$$
where $f:\widetilde{Y}_0(\gP)^{\tord} \to \widetilde{Y}^{\minor}$ is the restriction
of the composite of the extension of $\widetilde{\pi}_1$ with the projection
$\widetilde{Y}^{\tor} \to \widetilde{Y}^{\min}$.  Since the restriction of $\widetilde{\pi}_1$
to $\widetilde{Y}_0(\gP)^{\ord} \to \widetilde{Y}^{\ord}$ is finite, we may identify
$\widetilde{Y}_0(\gP)^{\ord}$ with an open subscheme of $\widetilde{Y}_0(\gP)^{\minor}$
and define $\widetilde{Y}_0(\gP)^{\min}$ by gluing $\widetilde{Y}_0(\gP)^{\minor}$ to
$\widetilde{Y}_0(\gP)$ along $\widetilde{Y}_0(\gP)^{\ord}$.  Just as for $\widetilde{Y}$
(i.e., the case $\gP = \CO_F$), we see that
$\widetilde{Y}_0(\gP)^{\min}$ is normal and independent of choice of cone decompositions,
its connected components are flat and projective over $\CO$, and we have a commutative diagram
$$\xymatrix{&\widetilde{Y}_0(\gP)^{\tor}\ar[dd]  \\
\widetilde{Y}_0(\gP) \ar[rd]\ar[ru] & \\
& \widetilde{Y}_0(\gP)^{\min} }$$
over the corresponding one for $\widetilde{Y}$, the
vertical (resp.~diagonal) morphism(s) being projective
(resp.~open immersions).

We can again identify our minimal compactification $\widetilde{Y}_0(\gP)^{\min}$ with an infinite disjoint union of the schemes denoted $\vec{\M}_{\CH,\CO}^{\spl,\min}$ in \cite{KWL:spl}.
Indeed since $\vec{\M}_{\CH,\CO}^{\spl,\min}$ is normal, we have
$\SPEC(g_*\CO_{\vec{\M}_{\CH,\Sigma,\CO}^{\spl,\tor}})
 = \vec{\M}_{\CH,\CO}^{\spl,\min}$, where $g$ is the projective
morphism $\vec{\M}_{\CH,\Sigma,\CO}^{\spl,\tor} \to
 \vec{\M}_{\CH,\CO}^{\spl,\min}$, and the morphism 
$\vec{\M}_{\CH,\CO}^{\spl,\min} \to \widetilde{Y}^{\min}$ is finite
over the ordinary locus, so we may identify $\widetilde{Y}_0(\gP)^{\minor}$ with the ordinary locus of the infinite disjoint union of the $\vec{\M}_{\CH,\CO}^{\spl,\min}$.

\subsection{Minimal compactification of ${Y}_0(\gP)$} \label{ss:Y0min}
The action of $\CO_{F,(p),+}^\times/(U\cap \CO_F^\times)$ on $\widetilde{Y}_0(\gP)$ again extends to $\widetilde{Y}_0(\gP)^{\min}$, and we define $Y_0(\gP)$ to be the quotient.

\begin{theorem} \label{thm:Y0min}
There is a normal scheme $Y_0(\gP)^{\min}$ over $\CO$
and an open immersion $Y_0(\gP) \hookrightarrow Y_0(\gP)^{\min}$
with the following properties:
\begin{enumerate}
\item $Y_0(\gP)^{\min}$ is normal, flat and projective over $\CO$, and independent of the choice of cone decompositions in its construction, and there is a canonical isomorphism between the reduced complement of $Y_0(\gP)$ in $Y_0(\gP)^{\min}$ and $\coprod_{c\in C_0(\gP)} Z_c$, where each $Z_c$ is isomorphic to $\Spec(\CO)$;
\item the completion of $Y_0(\gP)^{\min}$ along $Z_c$ is isomorphic to $\Spf(P_c)$, where
$$P_c =
\left\{\, \left. \sum_{m \in N^{-1}M_+ \cup \{0\}}\!\!\!\!\!  t_mq^m \,
\right| \, t_{\nu m} = t_m \in \CO\,\, \forall\,\, m \in N^{-1}M_+, \nu \in U \cap \CO_{F,+}^\times \,\right\}$$
for $c = [\underline{H}_1,\underline{H}_2,\alpha]$ and $M=\gd^{-1}I_1^{-1}J_2$;
\item the morphisms $\pi_i$ extend to $Y_0(\gP)^{\min} \to Y^{\min}$, with the restriction to the complement of $Y_0(\gP)$ being $\pi_i^{\infty}$ and the completion at $Z_c$ being the inclusion of local rings obtained by replacing $M$ by $M_i = \gd^{-1}I_i^{-1}J_i \stackrel[\alpha]{\sim}{\longrightarrow} \gq_i M$; in particular, $\pi_i$ is \'etale in a neighborhood of each cusp corresponding to a pair $([\underline{H}_i],\gq_i)$ such that $\gq_i = \CO_F$.
\end{enumerate}
\end{theorem}

The assertions in the theorem all follow from analogous ones with $Y_0(\gP)^{\min}$ replaced by $\widetilde{Y}_0(\gP)^{\min}$.  The first two parts can then be proved by minor modifications of the arguments in \cite[\S8]{chai} (see also \cite[\S4]{Dim}) or seen as a particular case of \cite[Thm.~4.3.1]{KWL:spl} (in which   
$\vec{\mathsf{Z}}^{\spl}_{[(\Phi_{\CH},\delta_{\CH})]} =
\vec{\mathsf{Z}}_{[(\Phi_{\CH},\delta_{\CH})]} = \Spec(\CO)$ for each $\Phi_{\CH} \neq 0$).  We note that, as in the discussion following \cite[(18)]{theta}, the isomorphism in (2) depends on a choice of splittings of the exact sequences $0 \to I_i \to H_i \to J_i \to 0$ for $i=1,2$, which we take to be compatible with $\alpha$.
Part (3) is then immediate from the construction of $\widetilde{Y}_0(\gP)^{\min}$ and part (3) of Theorem~\ref{thm:tY0tor}.

More generally, for any sufficiently small level $U$ and $N$ prime to $p$ such that $U' = U(N) \subset U$, the action of $U/U'$ on $Y'_0(\gP)$ extends to $Y'_0(\gP)^{\min}$, and we define $Y_0(\gP)^{\min}$ to be the quotient.  The resulting scheme is then independent of the choice of $N$ (up to changing the base $\CO$), and Theorem~\ref{thm:Y0min} holds exactly as stated above, except that the general description of the completed local ring $P_c$ is slightly more complicated (see Proposition~\ref{prop:koecher} below).  More precisely, it is given by the expression in the discussion preceding \cite[Prop.~7.2.1]{theta}, but with $\calC \in C'_0(\gP)$ lying over $c \in C_0(\gP)$ and $J^{-1}I$ replaced by $J_2^{-1}I_1$ in the definition\footnote{Note that the roles of $U$ and $U'$ are reversed here with respect to \cite{theta}, and that $\Gamma_{\calC} \subset \Aut_{\CO_F}(J_2 \times I_1)$ may be identified with
 the intersection of the groups similarly defined for the cusps 
$\underline{H}_i = \pi_i^{\infty}(\calC)$ of $C'$, or more precisely their images in $\Aut_{\CO_F}(J_2 \times I_1)\otimes \Q$ under the isomorphisms induced by $\alpha$.}
 of the group $\Gamma_{\calC}$ in \cite[(24)]{theta}.  
The resulting group $\Gamma_{\calC,U}$ appearing in the expression is thus isomorphic to $\Aut(\unH_1,\unH_2,\alpha)$, with the isomorphism depending on the choice of (compatible) splittings,
and hence to $B(F)_+ \cap g U_0(\gP) g^{-1}$ if $[\unH_1,\unH_2,\alpha]$ corresponds to the double coset 
$B(F)_+ g U_0(\gP)$.

\subsection{Hecke action}
Suppose that $g \in \GL_2(\A_{F,\f}^{(p)})$, $U$ and $U'$ are any sufficiently small (prime-to-$p$) levels such that $U' \subset gUg^{-1}$.  We then have the morphism 
$$\widetilde{\rho}_g:\widetilde{Y}'_0(\gP) \to \widetilde{Y}_0(\gP)$$
defined by the data $(\underline{A}_1,\underline{A}_2,\alpha)$, where if $(\underline{A}'_1,\underline{A}'_2,\alpha')$ is the universal triple over $\widetilde{Y}_0'(\gP)$, then the abelian schemes $A_i$ (for $i=1,2$) are characterized by the existence of a prime-to-$p$ quasi-isogeny $\psi_i:A_i' \to A_i$ such that the composite
$$(\A_{F,\f}^{(p)})^2
  \stackrel{\cdot g^{-1}}{\longrightarrow} (\A_{F,\f}^{(p)})^2
  \stackrel{\eta_i'}{\longrightarrow} T^{(p)}(A'_{i,\overline{s}})\otimes \Q
\stackrel{\psi_i}{\longrightarrow} T^{(p)}(A_{i,\overline{s}})\otimes \Q$$
induces an isomorphism $\eta_i: (\widehat{\CO}_F^{(p)})^2 \stackrel{\sim}{\to} \gd\otimes_{\CO_F} T^{(p)}(A_{i,\overline{s}})$ for all geometric points $\overline{s}$ of $\widetilde{Y}'_0(\gP)$, with the rest of the data defining 
$(\underline{A}_1,\underline{A}_2,\alpha)$ determined by the obvious compatibilities between the pair $(\psi_1,\psi_2)$ and the triple 
$(\underline{A}'_1,\underline{A}'_2,\alpha')$.
The resulting morphism $\widetilde{\rho}_g$ thus descends to a morphism $\rho_g:Y'_0(\gP) \to Y_0(\gP)$ giving rise to the map
on complex points
$$\GL_2(F)_+\backslash(\uhp^\Theta\times \GL_2(\A_{F,\f})/U'_0(\gP)) \longrightarrow 
\GL_2(F)_+\backslash(\uhp^\Theta\times \GL_2(\A_{F,\f})/U_0(\gP))$$
induced by right multiplication by $g$;
in particular $\rho_g\circ\rho_{g'} = \rho_{g'g}:Y''_0(\gP) \to Y_0(\gP)$ if 
$U'' \subset g'U'(g')^{-1}$.

Similarly (but more simply), there is a map 
$\widetilde{C}'_0(\gP) \to \widetilde{C}_0(\gP)$ sending $[\utH'_1,\utH'_2,\alpha']$ to the triple $[\utH_1,\utH_2,\alpha]$ such that $H_{i,(p)} = H'_{i,(p)}$ and $\eta_i' = \eta_i \circ g$ for $i=1,2$, where 
$g$ denotes right-multiplication by $g^{-1}$ and the rest of the data is determined by the obvious compatibilities.  Again this
descends to the map $\rho_g^{\infty}: C'_0(\gP) \to C_0(\gP)$ induced by right multiplication by $g$ on double cosets and satisfying the usual compatibility relation for varying $g$, $U$ and $U'$.

We claim that $\rho_g$ extends via $\rho_g^{\infty}$ to a morphism $Y'_0(\gP)^{\min} \to Y_0(\gP)^{\min}$ whose completion at the cusps has a simple description.  In order to make this precise, first recall that if $c = [\underline{H}_1,\underline{H}_2,\alpha]\in C_0(\gP)$ (resp.~$c' = [\underline{H}'_1,\underline{H}'_2,\alpha']\in C'_0(\gP)$), then a choice of splittings $\sigma_i: H_i \stackrel{\sim}{\to} J_i \times I_i$ compatible with  $\alpha$ (resp.~ $\sigma'_i: H'_i \stackrel{\sim}{\to} J'_i \times I'_i$ compatible with  $\alpha'$) is implicit in the isomorphism of Theorem~\ref{thm:Y0min}(2) describing the completion of $Y_0(\gP)^{\min}$ along $Z_c$
(resp.~$Y_0'(\gP)^{\min}$ along $Z'_{c'}$).
If $c = \rho_g^{\infty}(c')$, then $I_{i,(p)} = I'_{i,(p)} \subset H_{i,(p)} = H'_{i,(p)}$,
$\sigma_i = \smat{1}{\epsilon_i}{0}{1} \circ \sigma'_i$
for some
$\epsilon_i \in M^*_{i,(p)} = \Hom_{\CO_F}(J_i,I_{i,(p)})$
(with notation as in Theorem~\ref{thm:Y0min}(3)) and the matrix acting on $(J_i \times I_i)_{(p)}$ by right-multiplication), and the compatibilities with $\alpha$ and $\alpha'$ imply that $\epsilon_1 = \epsilon\circ \alpha$ and $\epsilon_2 = \alpha \circ \epsilon$ for some 
$$\epsilon \in M^*_{(p)} = \Hom_{\CO_F}(J_2,I_{1,(p)}).$$

\begin{lemma} \label{lem:hecke} 
With notation as above, the morphism 
$\rho_g:Y'_0(\gP) \to Y_0(\gP)$ extends uniquely to a morphism 
$Y'_0(\gP)^{\min} \to Y_0(\gP)^{\min}$.
Its restriction to the complement of 
$Y'_0(\gP)$ corresponds under Theorem~\ref{thm:Y0min}(1) to the morphism defined by $\rho_g^{\infty}$, and the resulting morphisms
on completions correspond under Theorem~\ref{thm:Y0min}(2) to the ones defined by
$$\sum_{m \in N^{-1}M_+ \cup \{0\}}\!\!\!\!\! t_m q^m \mapsto
   \sum_{m \in N^{-1}M_+ \cup \{0\}} \!\!\!\!\!
\zeta_{N'}^{-\epsilon(N'm)} t_m q^m$$
for any $N,N'$ such that $U(N) \subset U$ and $U(N') \subset U' \cap gU(N)g^{-1}$.
\end{lemma}
Before discussing the proof, we note that the hypotheses on $N$ and $N'$ imply that
$$N'\,\End_{\CO_F}(J_2' \times I_1') \smat{1}{\epsilon}{0}{1}
 \subset N \smat{1}{\epsilon}{0}{1}\End_{\CO_F}(J_2 \times I_1)$$
(as matrices acting via right multiplication), which implies that
$$N'\,\Hom_{\CO_F}(J_2',I_1') \subset N\,\Hom_{\CO_F}(J_2,I_1)\quad\mbox{and}\quad N'\epsilon \in N\,\Hom_{\CO_F}(J_2,I_1).$$
It follows that $N'M \subset NM'$ and $N'\epsilon \in NM^*$, so the formula in the statement does indeed define a map on the power series rings appearing in the statement of Theorem~\ref{thm:Y0min}.  Furthermore we have 
$ \smat{1}{-\epsilon}{0}{1} \Gamma_{\calC',U'} \smat{1}{\epsilon}{0}{1} \subset \Gamma_{\calC,U}$ as subgroups of $\Aut_{\CO_F}(J_2 \times I_1)$ acting via right multiplication, so the formula defines a map on the completions  as described for arbitrary $U$ and $U'$.  (Note also that taking $g=1$ in the proposition describes how the choice of splittings affects our uniformizations of the completions.)

To prove the lemma, we may assume that $U' = U(N')$ and $U = U(N)$.  The assertions then follow from analogous ones for toroidal compactifications, namely that the cone decompositions in their construction may be chosen so that the morphism $\widetilde{\rho}_g$  of $\widetilde{Y}_0(\gP)'$ extends to $\widetilde{Y}_0(\gP)^{\prime\tor} \to \widetilde{Y}_0(\gP)^{\tor}$ with the desired effect on (completions at) components of the complement.  To that end, we may assume that if $c = \rho_g^{\infty}(c')$, then the cone decomposition of $\Hom(M',\R)_{\ge 0} = \Hom(M,\R)_{\ge 0}$ chosen for $c'$ refines the one for $c$.  The assertions then reduce (by \cite[Thm.~3.4.1(4)]{KWL:spl}) to the claim that if $V$ is a complete DVR with fraction field $L$ and $s' \in \widetilde{Y}'_0(\gP)(L)$ is defined by 
$$(\gd^{-1}I'_1 \otimes \G_m)/\widetilde{q}^{\gd^{-1}J'_1} \longrightarrow (\gd^{-1}I'_2 \otimes \G_m)/(\widetilde{q}')^{\gd^{-1}J'_2}$$
for some $q':(N')^{-1}M' \to L^\times$ such that $\val\circ q' \in \Hom(M',\R)_{> 0}$ (with auxiliary data given by $c'$ and the splittings $\sigma_i'$), then $\widetilde{\rho}_g\circ s' = s$ where $s$ is similarly defined by (\ref{eqn:Tate}) with $q:N^{-1}M \to L^\times$ defined by $q^m = \zeta_{N'}^{-\epsilon(N'm)}q'^m$ (and auxiliary data given by $c$ and the $\sigma_i$).  Finally the claim itself is straightforward to check in view of the commutativity of the diagram
$$\xymatrix{(\A_{F,\f}^{(p)})^2
  \ar[r]^-{\sigma_i'\circ\eta_i'}
  \ar[d]^{\cdot g^{-1}} &
  (J_i' \times I_i')\otimes_{\CO_F}\A_{F,\f}^{(p)}
  \ar[r]^-{\sim} 
  \ar[d]^{\cdot\smat{1}{\epsilon_i}{0}{1}} &
  \gd \otimes_{\CO_F} T_i'  \ar[d]^{\wr} \\
  (\A_{F,\f}^{(p)})^2
  \ar[r]^-{\sigma_i\circ\eta_i}&
  (J_i \times I_i)\otimes_{\CO_F}\A_{F,\f}^{(p)}
  \ar[r]^-{\sim} &
  \gd \otimes_{\CO_F} T_i }$$
for $i=1,2$, where $T_i'$ is the prime-to-$p$ adelic Tate module of $(\gd^{-1}I_i' \otimes \overline{L}^\times)/\widetilde{q}^{\gd^{-1}J_i'}$,
the top right horizontal isomorphism is defined by $(x,y) \mapsto \widetilde{q}'^{x/N''}(\zeta_{N''} \otimes y)$ for 
$(x,y) \in (J'_i \times I'_i) \otimes \Z/N''\Z$ and $p \nmid N''$,
$T_i$ and the bottom right isomorphism are defined similarly, and the right vertical isomorphism is induced by the canonical quasi-isogeny $(\gd^{-1}I'_i \otimes \G_m)/\widetilde{q}^{\gd^{-1}J'_i} \longrightarrow (\gd^{-1}I_i \otimes \G_m)/(\widetilde{q}')^{\gd^{-1}J_i}$.

\subsection{The Koecher Principle}{\label{ss:KP}
Consider now the line bundle $\CA_{\vec{k},\vec{m},R}$ on $Y_R$, 
where as usual $\vec{k},\vec{m} \in \Z^{\Theta}$ and $R$ is a Noetherian $\CO$-algebra such that $\chi_{\vec{k}+2\vec{m},R} = 1$ on $\CO_F^\times \cap U$.  For $i=1,2$, we let $\CA^{(i)}_{\vec{k},\vec{m},R}$ denote its pull-back to $Y_0(\gP)_R$ via $\pi_i$
(omitting the subscript $R$ from the morphisms when clear from the context).  We explain how the Koecher Principle applies to describe $j_*\CA^{(i)}_{\vec{k},\vec{m},R}$, where $j$ denotes the open immersion $Y_0(\gP)_R \hookrightarrow Y_0(\gP)_R^{\min}$.

First assume $U = U(N)$ and consider the line bundle $\widetilde{\CA}^{(i)}_{\vec{k},\vec{m},R} = \widetilde{\pi}_i^*\widetilde{\CA}_{\vec{k},\vec{m},R}$ on $\widetilde{Y}_0(\gP)_R$.  Its extension $\widetilde{\CA}^{(i),\tor}_{\vec{k},\vec{m},R} := (\widetilde{\pi}_i^{\tor})^*\widetilde{\CA}^{\tor}_{\vec{k},\vec{m},R}$ to $\widetilde{Y}_0(\gP)^{\tor}$ is then formally canonical, so we may apply \cite[Thm.~4.4.10]{KWL:spl} to  identify $\widetilde{j}_* \widetilde{\CA}^{(i)}_{\vec{k},\vec{m},R}$ with the direct image of $\widetilde{\CA}^{(i),\tor}_{\vec{k},\vec{m},R}$ under the projection $\widetilde{Y}_0(\gP)^{\tor} \to \widetilde{Y}_0(\gP)^{\min}$.  In particular $\widetilde{j}_* \widetilde{\CA}^{(i)}_{\vec{k},\vec{m},R}$ is coherent and its completion along the cusp corresponding to $(\underline{H}_i,\gq_i)$ is described as in \cite[(22)]{theta}, taking $M = \gd^{-1}I_1^{-1}J_2$ and using $I_i$ (resp.~$J_i$) in place of $I$ (resp.~$J$) in the definition of the free rank one $\CO$-module $D_{\vec{k},\vec{m}}$ in \cite[(20)]{theta}.
The analogous statements then follow for $j_*\CA^{(i)}_{\vec{k},\vec{m},R}$ and any sufficiently small $U$,  in the sense that $j_*\CA^{(i)}_{\vec{k},\vec{m},R}$ is coherent and that its completion along each cusp is described as in \cite[Prop.~7.2.1]{theta}, with the evident modifications.  Furthermore, we may
describe the effect on $q$-expansions of the composite morphisms
\begin{equation}\label{eqn:minhecke}
j_* \CA^{(i)}_{\vec{k},\vec{m},R} \longrightarrow
   j'_* \rho_g^*\CA^{(i)}_{\vec{k},\vec{m},R} \longrightarrow 
      j'_*\CA'^{(i)}_{\vec{k},\vec{m},R}\end{equation}
in the setting of Proposition~\ref{lem:hecke},
where the latter map is induced by the quasi-isogeny
$\psi_i:A_i' \to A_i$ in the definition of $\rho_g$.
More precisely, letting $D^{(i)}_{\vec{k},\vec{m}}$
denote the invertible $\CO$-module
$$(I_i^{-1})_\theta^{\otimes k_\theta} 
   \otimes (\gd(I_iJ_i)^{-1})_\theta^{\otimes m_\theta}$$
(in the notation of \cite[(2)]{theta}), we have the following:

\begin{proposition} \label{prop:koecher} 
Suppose that $\vec{k},\vec{m} \in \Z^{\Theta}$, $U$ is a sufficiently small open compact subgroup of $\GL_2(\A_{F,\f})$ containing $U(N)$, and $R$ is a Noetherian $\CO$-algebra such that $\chi_{\vec{k}+2\vec{m},R} = 1$ on $\CO_F^\times \cap U$.  Then for $i=1,2$,
\begin{enumerate}
\item $j_*\CA^{(i)}_{\vec{k},\vec{m},R}$ is a coherent sheaf on $Y_0(\gP)_R^{\min}$;
\item for each $c = [\unH_1,\unH_2,\alpha ] \in C_0(\gP)$, the completion of $j_*\CA^{(i)}_{\vec{k},\vec{m},R}$ along $Z_c$ corresponds to the $P_{R,c}$-module
$Q^{(i)}_{\vec{k},\vec{m},R,c} :=   $
$$\left \{\, \sum_{m \in N^{-1}M_+ \cup \{0\}}  b \otimes r_mq^m \,\left|\, \begin{array}{c}
   r_{\alpha^{-1}\delta m} =  \zeta_N^{-\beta(\alpha^{-1}Nm)}
 \chi_{\vec{m},R} (\alpha)\chi_{\vec{k}+\vec{m},R}(\delta) r_m\\ \mbox{\rm for all $m \in N^{-1}M_+,
 \smallmat\alpha\beta0\delta \in \Gamma_{\calC,U}$} \end{array}
 \,\right.\right\},$$
where $b$ is a basis for $D^{(i)}_{\vec{k},\vec{m}}$ and $P_{R,c} = Q^{(i)}_{\vec{0},\vec{0},R,c} = H^0(\widehat{Y}_0(\gP)^{\min}_{Z_c,R},\CO_{\widehat{Y}_0(\gP)^{\min}_{Z_c,R}} )$ is as in Theorem~\ref{thm:Y0min}(2)
(or more precisely its variant for more general $U$ and $R$,
so in particular $\Gamma_{\calC,U} \cong \Aut(\unH_1,\unH_2,\alpha)$);

\item for $U'$ as above and $g \in \GL_2(\A_{F,\f}^{(p)})$
such that $U' \subset gUg^{-1}$, the morphism (\ref{eqn:minhecke}) corresponds under the isomorphisms in (2) to the map defined by
$$\sum_{m \in N^{-1}M_+ \cup \{0\}}  b \otimes r_mq^m \quad
\mapsto \quad \sum_{m \in N^{-1}M_+ \cup \{0\}}  b \otimes \zeta_{N'}^{-\epsilon(N'm)}r_mq^m $$
where $D'^{(i)}_{\vec{k},\vec{m}} = D^{(i)}_{\vec{k},\vec{m}}$ via the identification $H'_{i,(p)} = H_{i,(p)}$, and $N'$ and $\epsilon$ are as in the statement of Lemma~\ref{lem:hecke}.

\end{enumerate}
\end{proposition}

\section{Level $U_1(\gP)$}\label{sec:U1}
We now consider a more refined level structure at primes over $p$; more precisely, we let
$$U_1(\gP) = \left\{\left.\,\left(\begin{array}{cc}a&b\\c&d\end{array}\right) \in U_0(\gP) \,\right|\,\mbox{$d_\gp - 1 \in \gp\CO_{F,\gp}$ for all $\gp|\gP$}\,\right\}.$$
We follow \cite{DS1} to construct toroidal and minimal compactifications of $p$-integral models of the resulting Hilbert modular varieties, defined using a moduli problem based on \cite{pappas}, thus diverging from those obtained by the methods of \cite{KWL:spl}.

We assume throughout this section that $U$ is $\gP$-neat in the sense that 
$$U \cap \CO_F^\times = U_1(\gP) \cap \CO_F^\times.$$
Note that this holds for example if $U = U_1(\gn)$ for sufficiently small $\gn$.

\subsection{The model for $Y_1(\gP)$}\label{ss:Y1p}
Maintaining the notation from the preceding sections,  we let
$\psi:A_1 \to A_2$ denote the universal isogeny over $\widetilde{Y}_0(\gP)$, and write $G = \prod_{\gp|\gP} G_\gp$
where $G = \gd\gP\otimes_{\CO_F}\ker(\psi)$ and each $G_\gp$ is a Raynaud $(\CO_F/\gp)$-vector space scheme.
Consider also the dual isogeny $\psi^\vee:A_2^\vee \to A_1^\vee$ over $\widetilde{Y}_0(\gP)$, and identify its
kernel with $\gd\gP\otimes_{\CO_F} G^\vee$, where $G^\vee = \prod_{\gp|\gP} G_\gp^\vee$ is the Cartier dual of $G$.

The scheme $\widetilde{Y}_1(\gP)$ is then a certain closed subscheme of $G$, finite and flat over $\widetilde{Y}_0(\gP)$
of degree $\#(\CO_F/\gP)^\times$ (see \cite[\S5.1]{KS}).  As usual, we let $Y_1(\gP)$
denote its quotient by the action of $\CO_{F,(p),+}^\times$, so that $Y_1(\gP)$ is a model for
the Hilbert modular variety of level $U_1(\gP)$ and the morphism $Y_1(\gP) \to Y_0(\gP)$ corresponds
to the usual projection on complex points.

The neatness hypothesis implies that $Y_1(\gP) \to Y_0(\gP)$ is again finite and flat, and hence that $Y_1(\gP)$ is Cohen--Macaulay over $\CO$.  However our model for $Y_1(\gP)$ is not normal (unless $\gP = \CO_F$), and does not preserve the symmetry between the projections $\pi_1$ and $\pi_2$.  In particular the automorphism $w_\gP$ of $Y_0(\gP)$ does not lift to an automorphism of $Y_1(\gP)$, as follows for example from the fact that if the fibre over
$s \in Y_0(\gP)(\Fpbar)$ is \'etale, then that of $w_\gP(s)$ is 
connected.
	
\subsection{Cusps and clasps}
Just as for level $U_0(\gP)$, we define the set of {\em cusps} of level $U_1(\gP)$ to be
$$C_1(\gP) := B(F)_+\backslash \GL_2(\A_{F,\f}^{(p)})/U_1(\gP).$$
We thus have a bijection between $C_1(\gP)$ and the set of isomorphism classes of pairs $(\underline{H},\Psi)$, where $\underline{H}$ corresponds to an element of $C = C_U$ and $\Psi \in H/\gP H$ is such that $\Ann_{\CO_F} \Psi = \gP$.  (The bijection is induced by $g \mapsto (\underline{H},\Psi)$, where
$H =  \widehat{\CO}_F^2 g^{-1} \cap F^2$, $I = H \cap (0\times F)$, $\lambda = \det$, $\eta$ is induced by right-multiplication by $g^{-1}$ and $\Psi$ is the image of $(0,1)\cdot g^{-1}$.)

As before, the set $C_1(\gP)$ will parametrize the complement of $Y_1(\gP)_K$ in its minimal compactification; now however collections of cusps coalesce in connected components of the complement on the integral model.
With this in mind, we define a {\em clasp} of level $U_1(\gP)$ to be an isomorphism class of pairs $(\underline{H},\Xi)$, where 
$\underline{H}$ is as in the description of $C$ and $\Xi \in J/\gP J$ (with $J = H/I$ and the obvious notion of isomorphism).
We let $C_{1/2}(\gP)$ denote the set of clasps of level $U_1(\gP)$, so we have the natural degeneracy maps
\begin{equation} \label{eqn:clasps} \begin{array}{ccccccc}
C_1(\gP) &\longrightarrow & C_{1/2}(\gP) & \longrightarrow &
C_0(\gP) & \stackrel{\pi_1^\infty}{\longrightarrow} & C \\
  
[\underline{H},\Psi]& \mapsto &[\underline{H},\Xi] & \mapsto &
([\underline{H}],\gq) & \mapsto& \underline{H},\end{array}\end{equation}
where $\Xi$ is the image of $\Psi$, $\gq = \Ann_{\CO_F} \Xi$, and as usual $[\cdot]$ denotes the associated isomorphism class.

\subsection{Toroidal compactification}\label{ss:Y1ptor}

Once again we initially assume $U = U(N)$ for some integer $N$ (prime to $p$).  Let $\widetilde{Y}_0(\gP)^\tor$ be a toroidal compactification of $\widetilde{Y}_0(\gP)$, and write $\psi^\tor:A_1^\tor \to A_2^\tor$ for the extension of the universal isogeny $\psi$.

Fix for the moment a prime $\gp|\gP$ and write $\widetilde{C}_0(\gP) = \widetilde{C}_0(\gP)_1 \coprod \widetilde{C}_0(\gP)_2$ where $\widetilde{C}_0(\gP)_1$ 
(resp.~$\widetilde{C}_0(\gP)_2$) is the set of cusps
such that $\alpha(J_1) \subset \gp J_2$ (resp.~$\alpha(I_1) \subset \gp I_2$).  For $i=1,2$, we let $Z_i$ denote the preimage in $\widetilde{Y}_0(\gP)^\tor$ of the closed subscheme of 
$\widetilde{Y}_0(\gP)^{\min}$ corresponding to $\widetilde{C}_0(\gP)_i$, and
define $\widetilde{Y}_0(\gP)^\tor_i$ 
to be the complement 
in $\widetilde{Y}_0(\gP)^\tor$ of $Z_i$, so that
$$\widetilde{Y}_0(\gP)^\tor =
 \widetilde{Y}_0(\gP)^\tor_1 \cup \widetilde{Y}_0(\gP)^\tor_2
\quad\mbox{and}\quad
\widetilde{Y}_0(\gP) =
 \widetilde{Y}_0(\gP)^\tor_1 \cap \widetilde{Y}_0(\gP)^\tor_2.$$

We then have that $G_\gp$ extends uniquely over 
$\widetilde{Y}_0(\gP)^\tor_1$ to a finite flat subgroup scheme of $\gd\gP\otimes_{\CO_F}\ker(\psi^\tor)$, its completion along $Z_2$ being the subgroup corresponding to the image of 
$$(I_1/\gp I_1) \otimes \mu_p
 = ((I_1/\gP I_1) \otimes \mu_p)_\gp  \,\,
\stackrel{p}{\hookrightarrow} \,\,(\gP I_1 \otimes \G_m)/\widetilde{q}^{\gP J_1}.$$
Similarly $\gd\gP \otimes_{\CO_F} G_\gp^\vee$ extends uniquely over $\widetilde{Y}_0(\gP)^{\tor}_2$ to a finite flat subgroup scheme of the kernel of the extension of $\psi^\vee$, its completion along $Z_1$ being the subgroup corresponding to the image of $(\gp J_2)^{-1}/J_2^{-1} \otimes \mu_p $ in $(J_2^{-1} \otimes \G_m)/\widetilde{q}^{I_2^{-1}}$.  Taking Cartier duals, it follows that $G_\gp$ extends to a finite flat group scheme 
over $\widetilde{Y}_0(\gP)^\tor_2$ as well, and gluing to its extension over $\widetilde{Y}_0(\gP)^\tor_1$ thus yields an 
$(\CO_F/\gp)$-vector space scheme $G_\gp^\tor$ over 
$\widetilde{Y}_0(\gP)^\tor$.  We then let $G^\tor = \prod_{\gp|\gP} G_\gp^\tor$, and define $\widetilde{Y}_1(\gP)^\tor$ as the extension of $\widetilde{Y}_1(\gP)$ to a closed subscheme of $G^\tor$, finite and flat over $\widetilde{Y}_0(\gP)^\tor$.

\subsection{Minimal compactification}\label{ss:Y1pmin}
We define the minimal compactification of $\widetilde{Y}_1(\gP)$ as
$$\widetilde{Y}_1(\gP)^{\min} = 
\SPEC(f_*\CO_{\widetilde{Y}_1(\gP)^{\tor}}),$$
where $f$ is the composite
$\widetilde{Y}_1(\gP)^\tor \to \widetilde{Y}_0(\gP)^\tor \to \widetilde{Y}_0(\gp)^{\min}$. 
As usual $\widetilde{Y}_1(\gP)^{\min}$ is independent of the choice of cone decomposition in its definition, the action of
$\CO_{F,(p),+}^\times$ on $\widetilde{Y}_1(\gP)$ extends to it uniquely, and we let $Y_1(\gP)^{\min}$ denote the quotient by this action.  More generally, for any sufficiently small level $U$ ($\gP$-neat and prime to $p$) we define $Y_1(\gP)^{\min}$ as the quotient of $Y'_1(\gP)^{\min}$ by the unique extension of the action of $U/U'$ on $Y'_1(\gP)$, where $Y'_1(\gp)$ is defined using $U' = U(N) \subset U$ for suitable $N$ (of which the resulting scheme $Y_1(\gP)^{\min}$ is independent).

\begin{lemma} \label{lem:flat} For sufficiently small $U$, the morphism $h:Y_1(\gP)^{\min} \to Y_0(\gP)^{\min}$ is finite and flat.
\end{lemma}
\begin{proof} Since the morphism is finite and its restriction to $Y_1(\gP)$ is flat, it suffices to prove that its completion at each cusp of $Y_0(\gP)^{\min}$ is flat.  To that end, fix a cusp $c \in C_0(\gP)$, corresponding to the isomorphism class of $(\underline{H}_1,\underline{H}_2,\alpha)$, and write $\gP = \gq_1\gq_2$ as above, so $\gp|\gq_1$ (resp.~$\gp|\gq_2$) if and only if $\alpha(J_1) \subset \gp J_2$ (resp.~$\alpha(I_1) \subset \gp I_2$).  Suppose for example that
$$U \subset U_1^1(\gn) =
 \left\{\left.\,\left(\begin{array}{cc}a&b\\c&d\end{array}\right) \in \GL_2(\widehat{\CO}_F) \,\right|\, a \equiv d \equiv 1,\, c \equiv 0 \bmod \gn\widehat{\CO}_F\,\right\}$$
for some $\gn$ (prime to $p$) such that if $\mu \in \CO_F^\times$ and $\mu \equiv 1 \bmod \gn$, then $\mu \equiv 1 \bmod \gP$.  Note in particular that $U$ is $\gP$-neat, and hence so is $U' := U(N)$ for any $N$ such that $U(N) \subset U$. 
For $i=0,1$, let $\gY_i$ denote the completion of $\widetilde{Y}_i'(\gP)^\tor$ along the preimage of $\widetilde{c}$, where 
$\widetilde{c}\in \widetilde{C}'_0(\gP)$ is any cusp lying over $c$.  The construction of $\widetilde{Y}'_1(\gP)^{\tor}$ and the condition on $N$ then give an isomorphism 
\begin{equation} \label{eqn:complete} \gY_1 \cong \gY_0 \times_{\CO} T, \end{equation}
where $T = \prod_{\gp|\gP} T_\gp$ and $T_\gp$ is the finite flat $\CO$-algebra representing generators of the \'etale (resp.~multiplicative) $(\CO_F/\gp)$-vector space scheme
$$J_1/\gp J_1
\qquad \mbox{(resp. $\mu_p \otimes (I_1/\gp I_1) $\,)}$$
if $\gp|\gq_1$ (resp.~$\gp|\gq_2$).  It then follows from the Theorem on Formal Functions that (\ref{eqn:complete}) holds with ``tor'' replaced by ``min'' in the definition of $\gY_i$ for $i = 0,1$.  Furthermore the condition on $U$ implies that if
$$\gamma = \left(\begin{array}{cc}\alpha & \beta \\ 0 & \delta \end{array}\right) \in \Aut_{\CO_F}(J_2 \times I_1) \cap gUg^{-1}$$
for some $\widehat{\CO}_F$-linear isomorphism $g:\widehat{\CO}_F^2 \cong \widehat{J}_2 \times \widehat{I}_1$ (where the actions are by right multiplication), then $\alpha \equiv \delta \equiv 1 \bmod \gP$.  The isomorphism (\ref{eqn:complete}) is therefore compatible with the action of the stabilizer of $\widetilde{c}$ in $\CO_{F,(p),+}^\times \times (U/U')$, so it holds with $\widetilde{Y}'_i(\gP)^{\min}$ replaced by its quotient $Y_i(\gP)^{\min}$ for $i=0,1$.
\end{proof}

\subsection{Completions at clasps}
The proof of Lemma~\ref{lem:flat} also yields a description of the completion of $Y_1(\gP)^{\min}$ along the complement of $Y_1(\gP)$.  Maintaining the notation there, note that 
the hypothesis on $U$ ensures that every
automorphism of $(\underline{H}_1,\underline{H}_2,\alpha)$
fixes $H_2/\alpha(H_1)$ pointwise. The proof of the lemma shows that the  connected components of the fibre in $Y_1(\gP)^{\min}$ over $c \in C_0(\gP)$ correspond to generators of $J_1/\gq_1J_1$, hence to isomorphism classes of pairs
$(\underline{H}_1,\Xi)$ such that $\Ann_{\CO_F}(\Xi) = \gq_1$,
i.e., elements of $C_{1/2}(\gP)$ lying over $c$.
We thus obtain a bijection between $C_{1/2}(\gP)$ and the set
of connected components of the complement of $Y_1(\gP)$ in $Y_1(\gP)^{\min}$.

In particular for each cusp $[\underline{H}] = [H,I,[\lambda],[\eta]] \in Y^{\min}$ (for such $U$), the connected components of its preimage in $Y_1(\gP)^{\min}$ under the composite
$$\rho = \pi_1 \circ h: \,Y_1(\gP)^{\min} \longrightarrow Y_0(\gP)^{\min} {\longrightarrow} Y^{\min}$$
are in bijection with 
$J/\gP J$.  Furthermore the completion of $Y_1(\gP)^{\min}$ along the component corresponding to 
an element $\Xi \in J/\gP J$ has the same description as the completion of $Y^{\min}$ at $[\underline{H}]$, but with $M = \gd^{-1}I^{-1}J$ replaced by $\gq^{-1}M$ and
$\CO$ replaced by the finite flat local $\CO$-algebra $T_{\gor}$ representing  generators of $\mu_p \otimes (I/\gor I)$, where 
$\gq = \Ann_{\CO_F}\Xi$ and $\gP = \gq\gor$.
For example if $U = U(\gn)$ for some $\gn$
as in the proof of Lemma~\ref{lem:flat}, then the completion is isomorphic\footnote{Recall that the isomorphism depends on a choice of splitting of $0 \to I \to H \to J \to 0$, as made precise by Lemma~\ref{lem:hecke}.} to $\Spf(P)$ where
$$P = \left \{\, \sum_{m \in (\gq^{-1}\gn^{-1}M)_+ \cup \{0\}}  t_mq^m \,
\left|\, \begin{array}{c} \mbox{$t_{\nu m} = t_m \in T_{\gor}$ for all
$m \in (\gq^{-1}M)_+ \cup \{0\}$,} \\
 \nu \in \CO_{F,+}^\times,\,\, \nu \equiv 1 \bmod \gn\end{array}\,\right.\right\}.$$
We note also that the action of $U_0(\gP)/U_1(\gP) \cong (\CO_F/\gP)^\times$ on $Y_1(\gP)$ extends to $Y_1(\gP)^{\min}$ (with the obvious description on the completion at the fibre over a cusp of $Y^{\min}$), and we may identify $Y_0(\gP)^{\min}$ with the quotient. For completeness, we consider more general $U$ below, 
but first we record the following immediate consequence of Lemma~\ref{lem:flat}:
\begin{corollary} \label{cor:flat}
For sufficiently small $U$, the morphism $\rho = \pi_1\circ h$ is flat in a neighborhood of each component of the complement of $Y_1(\gP)$ in $Y_1(\gP)^{\min}$ corresponding to a clasp of the form $[\underline{H},0]$.  Similarly the composite $\pi_2 \circ h$ is flat, and in fact \'etale, in a neighborhood of each component corresponding to a clasp $[\underline{H},\Xi]$ such that $\CO_F\cdot\Xi = J/\gP J$.
\end{corollary}

Suppose now only that $U$ is $\gP$-neat and sufficiently small in the usual sense (as in \cite[\S2.2]{theta}).  Since $Y_1(\gP)^{\min}$ is the quotient of $ Y'_1(\gP)^{\min}$ by the action of $U/U'$ (where $U' = U(N)$), we see that the action of $U_0(\gP)/U_1(\gP) \cong (\CO_F/\gP)^\times$ on $Y_1(\gP)$ extends to $Y_1(\gP)^{\min}$, with $Y_0(\gP)^{\min}$ as the quotient. 
Furthermore taking $U$-invariants yields the following description of the completion of $Y_1(\gP)^{\min}$ along the complement of $Y_1(\gP)$ (as in \cite[\S7.2]{theta}):
\begin{proposition} \label{prop:Y1pcomp} Suppose that $U$ is $\gP$-neat, $U(N) \subset U$ and $\zeta_{pN} \in \CO$.
\begin{enumerate}
\item There is a bijection between $C_{1/2}(\gP)$ (resp.~$C_1(\gP)$) and the set of connected components of the complement of $Y_1(\gP)$ in $Y_1(\gP)^{\min}$ (resp.~$Y_1(\gP)_K$ in $Y_1(\gP)_K^{\min}$), under which (\ref{eqn:clasps}) is compatible (in the obvious sense) with the morphisms
$$Y_1(\gP)_K^{\min} \,\,\longrightarrow \,\, Y_1(\gP)^{\min}\,\, \longrightarrow\,\, Y_0(\gP)^{\min} \,\,  \stackrel{\pi_1}{\longrightarrow}\,\, Y^{\min} .$$
\item Let $[\unH,\Xi] \in C_{1/2}(\gP)$,
 $\gq = \Ann_{\CO_F}(\Xi)$, $\gor = \gq^{-1}\gP$, and
$$\Gamma = \left\{\,\left. \smat{\alpha}{\beta}{0}{\delta} \in \Gamma_{\calC,U} \,\right|\,  \alpha \equiv 1\bmod \gq,\,\, \beta \in \gq J^{-1} I\,\right\}$$
(in the notation following \cite[(24)]{theta} for any 
splitting
$H \isoto J \times I$ and $\calC  = [H,I,[\lambda],[\eta']] \in C_{U_0(N)}$ lying over $[\unH]$, so $\Gamma \cong \Aut(\unH,\Xi)$),
and write $t = \sum t^\xi$ for $t \in T_\gor = \bigoplus T_{\gor}^{\xi}$, where the decomposition is over characters $\xi:(\CO_F/\gor)^\times \to \CO^\times$ (so each $T_{\gor}^{\xi}$ is a free $\CO$-module of rank one).
The completion of $Y_1(\gP)^{\min}$ along the component corresponding to $[\unH,\Xi]$ is isomorphic to $\Spf(P)$, where $P = $
$$\left \{\, \sum_{m \in (\gq^{-1}N^{-1}M)_+ \cup \{0\}}\!\!\!\! \!\!\!\!\!\! t_mq^m \,
\left|\, \begin{array}{c} 
\mbox{$t_{\alpha^{-1}\delta m}^{\xi} = \zeta_N^{-\beta(\alpha^{-1} Nm)} \xi(\delta) t_m^{\xi}$ for all $\smat{\alpha}{\beta}{0}{\delta} \in \Gamma$,}
\\
m \in (\gq^{-1}N^{-1}M)_+ \cup \{0\},\,\,
\xi: (\CO_F/\gor)^\times \to \CO^\times\end{array}\,\right.\right\}$$
(as usual letting $M = \gd^{-1}I^{-1}J$ and viewing
$\beta: \gq^{-1} M \to \Z$ via the canonical isomorphism
$\gq J^{-1}I \isoto \Hom(\gq^{-1}M,\Z)$).
\item For $U'$ as above and $g \in \GL_2(\A_{F,\f}^{(p)})$
such that $U' \subset gUg^{-1}$, the morphism $\rho_g:Y'_1(\gP) \to Y_1(\gP)$ extends uniquely to a morphism $Y'_1(\gP)^{\min} \to Y_1(\gP)^{\min}$ over $Y'^{\min} \to Y^{\min}$.  
The extension is compatible in the obvious sense with the map 
$C'_{1/2}(\gP) \to C_{1/2}(\gP)$ induced by right multiplication by $g$, and the resulting map on completions is given by the same formula as in Lemma~\ref{lem:hecke}, where now $t_m \in T'_{\gor} = T_{\gor}$ via the identification $I' \otimes \mu_p = I \otimes \mu_p$.
\end{enumerate}
\end{proposition}

We remark that the assumption that $\zeta_p \in \CO$ is only made in order to incorporate the assertions about $Y_1(\gP)_K$ and $C_1(\gP)$ into the statement.

\subsection{$q$-expansions}
Suppose as usual that $\vec{k},\vec{m} \in \Z^{\Theta}$ and $R$ is a Noetherian $\CO$-algebra such that $\chi_{\vec{k}+2\vec{m},R} = 1$ on $\CO_F^\times \cap U$, and consider the line bundle
$$\CA^{\flat}_{\vec{k},\vec{m},R}
:= \rho^* \CA_{\vec{k},\vec{m},R}
= h^*\CA^{(1)}_{\vec{k},\vec{m},R}$$ 
on $Y_1(\gP)_R$ (writing $\rho$ and $h$ also for their restrictions to $Y_1(\gP)_R$ and omitting the subscripts $R$).
We claim that $j^{\flat}_*\CA^{\flat}_{\vec{k},\vec{m},R}$ is coherent and describe its completion along the complement of the image of $j^{\flat}:Y_1(\gP)_R\hookrightarrow Y_1(\gP)^{\min}_R$.
 Assume first that $U$ is of the form $U(N)$ and let $\widetilde{h}^{\tor}$ denote the finite flat morphism $\widetilde{Y}_1(\gP)^{\tor} \to \widetilde{Y}_0(\gP)^{\tor}$.  The isomorphism (\ref{eqn:complete}) then shows that the vector bundle $\widetilde{h}^{\tor}_*\CO_{\widetilde{Y}_1(\gP)^{\tor}}$ is formally canonical, and hence so is 
$$\widetilde{h}^{\tor}_* \widetilde{\CA}^{\flat,\tor}_{\vec{k},\vec{m},R} = \widetilde{\CA}^{(1),\tor}_{\vec{k},\vec{m},R}
\otimes \widetilde{h}^{\tor}_*\CO_{\widetilde{Y}_1(\gP)^{\tor}},
$$ where $\widetilde{\CA}^{\flat,\tor}_{\vec{k},\vec{m},R} = (\widetilde{h}^{\tor})^*\widetilde{\CA}^{(1),\tor}_{\vec{k},\vec{m},R}$.  We may therefore apply the Koecher Principle (again in the form of \cite[Thm.~4.4.10]{KWL:spl}) to conclude that
$f_*\widetilde{\CA}^{\flat,\tor}_{\vec{k},\vec{m},R} = \widetilde{h}_*\widetilde{j}_*^{\flat} \CA^{\flat}_{\vec{k},\vec{m},R}$, where $f$ is the morphism $\widetilde{Y}_1(\gP)_R^{\tor} \to \widetilde{Y}_0(\gP)_R^{\min}$.  Since $\widetilde{h}$ is finite, it follows that $\widetilde{j}_*^{\flat} \widetilde{\CA}^{\flat}_{\vec{k},\vec{m},R}$ is the direct image of $\widetilde{\CA}^{\flat,\tor}_{\vec{k},\vec{m},R}$ under 
$\widetilde{Y}_1(\gP)_R^{\tor} \to \widetilde{Y}_1(\gP)_R^{\min}$,
hence is coherent.  Taking quotients by group actions then yields the following for any sufficiently small (in particular $\gP$-neat) $U$:  
\begin{proposition} \label{prop:Y1pqexp}
The sheaf ${j}^{\flat}_*{\CA}^{\flat}_{\vec{k},\vec{m},R}$ is coherent, and if $Z$ is the connected component of the complement
of $Y_1(\gP)$ in $Y_1(\gP)^{\min}$ corresponding to the 
clasp $(\underline{H},\Xi) \in C_{1/2}(\gP)$,
then the completion of ${j}^{\flat}_*{\CA}^{\flat}_{\vec{k},\vec{m},R}$ along $Z_R$ is the sheaf associated to the $P$-module
$$\left \{\, \sum_{m \in (\gq^{-1}N^{-1}M)_+ \cup \{0\}}\!\!\!\! \!\!\!\!\!\! b \otimes t_mq^m \,
\left|\, \begin{array}{c} 
t_{\alpha^{-1}\delta m}^{\xi} = \zeta_N^{-\beta(\alpha^{-1} Nm)}\chi_{\vec{m},R}(\alpha)\chi_{\vec{k} + \vec{m},R}(\delta) \xi(\delta) t_m^{\xi}\\
\mbox{ for all $\smat{\alpha}{\beta}{0}{\delta} \in \Gamma',\,\,\,\,
m \in (\gq^{-1}N^{-1}M)_+ \cup \{0\}$,}\\
\xi: (\CO_F/\gor)^\times \to \CO^\times\end{array}\,\right.\right\},$$
where $b$ is any basis for $D_{\vec{k},\vec{m}}$ (in the notation of \cite[(20)]{theta}) and the rest of the notation is as in Propositon~\ref{prop:Y1pcomp}, except that now $t_m \in T_{\gor} \otimes_{\CO} R$ in the preceding expression and the definition of $P = P_R$.  Furthermore if $g\in \GL_2(\A_{F,\f}^{(p)})$ and $U' \subset gUg^{-1}$, then the maps induced by the morphisms $j'^{\flat}_*\CA'^{\flat}_{\vec{k},\vec{m},R} \to
j^{\flat}_*\CA^{\flat}_{\vec{k},\vec{m},R}$ on completions are given by the same formula as in Proposition~\ref{prop:koecher}(3).

\end{proposition}

\section{Kodaira--Spencer and cohomological vanishing} \label{sec:KS}

We now proceed to explain how the main results of \cite{KS} extend to toroidal compactifications.  We assume $U = U(N)$, and choose
polyhedral cone decompositions for cusps in $\widetilde{C}$ so
that the resulting toroidal compactification $\widetilde{Y}^{\tor}$ is smooth over $\CO$.  We let $\widetilde{Z}$ denote the reduced complement of $\widetilde{Y}$ in $\widetilde{Y}^{\tor}$.

\subsection{The Kodaira--Spencer isomorphism}
\label{ss:KS}

Recall from \cite[\S7.3]{theta} that the Kodaira--Spencer filtration on $\Omega^1_{\widetilde{Y}/\CO}$ extends to one on
$\Omega^1_{\widetilde{Y}/\CO}(\log\widetilde{Z})$, with graded pieces canonically isomorphic to
$\widetilde{\CA}_{2\e_\theta,-\e_\theta}$ (omitting the subscript $R$ when $R = \CO$).  It follows that
$$\CK_{\widetilde{Y}^{\tor}/\CO}(\widetilde{Z})
 = \Omega^d_{\widetilde{Y}^\tor/\CO}(\widetilde{Z})
 = \bigwedge\nolimits^d_{\CO_{\widetilde{Y}^\tor}}\left(\Omega^1(\log\widetilde{Z})\right)$$
is canonically isomorphic to $\widetilde{\CA}_{\vec{2},-\vec{1}} = \widetilde{\delta}^{-1}\widetilde{\omega}^{\otimes 2}$, where as usual
$\CK_{\widetilde{Y}^\tor/\CO}$ denotes the dualizing sheaf of $\widetilde{Y}^{\tor}$ over $\CO$.  Furthermore the extension is obtained from isomorphisms
\begin{equation}\label{eqn:canonical}\xi^* (\CK_{\widetilde{Y}^{\tor}/\CO}(\widetilde{Z}))
  \,\,\cong\,\, D_{F/\Q}^{-1}\bigwedge\nolimits^d(I^{-1}J) \otimes \CO_{\widehat{S}} \,\,\cong \,\,\xi^*(\widetilde{\delta}^{-1}\widetilde{\omega}^{\otimes 2}),\end{equation}
where $\widehat{S}$ is the formal scheme whose quotient 
by $(\CO_F^{\times}\cap U)^2$ defines the completion of $\widetilde{Y}^{\tor}$ along the connected component of $\widetilde{Z}$ corresponding to a cusp $(H,I,\lambda,[\eta])$, $\xi:\widehat{S}\to \widetilde{Y}^{\tor}$ is the composite of the quotient map with the completion, and the second isomorphism in (\ref{eqn:canonical}) is given by the canonical trivializations 
$$\xi^*\widetilde{\omega} \cong \bigwedge\nolimits^d(I^{-1})\otimes \CO_{\widehat{S}}\quad\mbox{and}\quad
\xi^*(\widetilde{\delta}^{-1}\widetilde{\omega}) \cong D_{F/\Q}^{-1}\bigwedge\nolimits^d J\otimes \CO_{\widehat{S}}.$$

Turning now to level $U_0(\gP)$, recall from \cite[Thm.~3.2.1]{KS} that the Kodaira--Spencer isomorphism on $\widetilde{Y}_0(\gP)$ takes the form 
\begin{equation} \label{eqn:KS0P}
\CK_{\widetilde{Y}_0(\gP)/\CO} \cong \widetilde{\pi}_1^*\widetilde{\omega} \otimes_{\CO_{\widetilde{Y}_0(\gP)}} \widetilde{\pi}_2^*(\widetilde{\delta}^{-1}\widetilde{\omega}).
\end{equation}
To extend this to toroidal compactifications, let us
choose the polyhedral cone decompositions for cusps in $\widetilde{C}_0(\gP)$ so that the morphisms $\widetilde{\pi}_i$ (for $i=1,2$) extend to $\widetilde{Y}_0(\gP)^\tor \to \widetilde{Y}^\tor$ (which we still denote $\widetilde{\pi}_i$). Furthermore since $\widetilde{Y}_0(\gP)^{\ord}$ is smooth, we may refine the chosen cone decompositions so as to ensure that 
$\widetilde{Y}_0(\gP)^{\tord}$ is smooth.  Note in particular 
that $\widetilde{Y}_0(\gP)^\tor$ is a local complete intersection over $\CO$.

\begin{theorem} \label{thm:KS}
The isomorphism (\ref{eqn:KS0P}) extends to an isomorphism
$$\CK_{\widetilde{Y}_0(\gP)^\tor}(\widetilde{Z}_0(\gP)) \cong \widetilde{\pi}_1^*\widetilde{\omega} \otimes_{\CO_{Y_0(\gP)^\tor}} \widetilde{\pi}_2^*(\widetilde{\delta}^{-1}\widetilde{\omega}),$$
where $\widetilde{Z}_0(\gP)$ is the reduced complement of $\widetilde{Y}_0(\gP)$ in $\widetilde{Y}_0(\gP)^\tor$.
\end{theorem}
\begin{proof}
The same argument that establishes (\ref{eqn:canonical})
gives an isomorphism
$$\xi_0^* (\CK_{\widetilde{Y}_0(\gP)^{\tor}/\CO}(\widetilde{Z}_0(\gP)))
 \cong D_{F/\Q}^{-1}\bigwedge\nolimits^d(I_1^{-1}J_2) \otimes \CO_{\widehat{S}_0(\gP)} \cong \xi_0^*(\widetilde{\pi}_1^*(\widetilde{\omega})\otimes_{\CO_{\widehat{Y}_0(\gP)^\tor}}\widetilde{\pi}_2^*(\widetilde{\delta}^{-1}\widetilde{\omega})),$$
where $\widehat{S}_0(\gP)$ is the formal scheme whose quotient defines the completion along a connected component corresponding to a cusp of the form $(\underline{H}_1,\underline{H}_2,\alpha)$, and
$\xi_0:\widehat{S}_0(\gP) \to \widetilde{Y}_0(\gP)^\tor$ is the resulting morphism.  Furthermore the isomorphisms are compatible with those of (\ref{eqn:canonical}) under $\widetilde{\pi}_1$ in the sense that the diagram
$$\begin{array}{ccccc}
\xi_0^*\widetilde{\pi}_1^* (\CK_{\widetilde{Y}^{\tor}/\CO}(\widetilde{Z}))
  & \cong & D_{F/\Q}^{-1}\bigwedge\nolimits^d(I_1^{-1}J_1) \otimes \CO_{\widehat{S}_0(\gP)} &\cong &\xi_0^*\widetilde{\pi}_1^*(\widetilde{\delta}^{-1}\widetilde{\omega}^{\otimes 2})
\\
\downarrow&&\downarrow&&\downarrow \\
\xi_0^* (\CK_{\widetilde{Y}_0(\gP)^{\tor}/\CO}(\widetilde{Z}_0(\gP)))
& \cong & D_{F/\Q}^{-1}\bigwedge\nolimits^d(I_1^{-1}J_2) \otimes \CO_{\widehat{S}_0(\gP)}& \cong &\xi_0^*(\widetilde{\pi}_1^*(\widetilde{\omega})\otimes\widetilde{\pi}_2^*(\widetilde{\delta}^{-1}\widetilde{\omega}))
\end{array}
$$
commutes, where the top line is the pull-back of (\ref{eqn:canonical}) via the morphism $\zeta:\widehat{S}_0(\gP) \to \widehat{S}$ over $\widetilde{\pi}_1$, the leftmost arrow is induced by the morphism $\widetilde{\pi}_1^*\CK_{\widetilde{Y}^\tor/\CO} \to \CK_{\widetilde{Y}_0(\gP)^\tor/\CO}$, the middle by the morphism $J_1 \to J_2$ defined by $\alpha$, and the last by the morphism 
$\widetilde{\pi}_1^*(\widetilde{\delta}^{-1}\widetilde{\omega}) \to \widetilde{\pi}_2^*(\widetilde{\delta}^{-1}\widetilde{\omega})$
induced by the extension of the universal isogeny to the semi-abelian schemes over $\widetilde{Y}_0(\gP)^\tor$. 
The theorem then follows from the compatibility under $\widetilde{\pi}_1$ between the Kodaira--Spencer isomorphisms on $\widetilde{Y}_0(\gP)$ and $\widetilde{Y}$ (see \cite[Prop.~3.2.3]{KS}), 
the completion of the desired isomorphism along each connected component of $\widetilde{Z}_0(\gP)$ being the one described at the start of the proof
(or more precisely, its quotient by $(\CO_F^\times \cap U)^2$).
\end{proof}

\subsection{Cohomological vanishing} \label{ss:vanish}

Now we explain how the cohomological vanishing results (and consequences) of \cite[\S5.3]{KS} extend to toroidal compactifications.  Recall from \cite[Cor.~5.3.2, 5.3.5]{KS}
that the coherent sheaves
$$R^i\pi_{1,*}\CK_{Y_0(\gP)/\CO}\quad\mbox{and}\quad R^i\pi_{1,*}\CO_{Y_0(\gP)}$$
vanish for $i > 0$ and are locally free if $i=0$, so the same holds with $Y_0(\gP)$ (resp.~$\pi_1$) replaced by $\widetilde{Y}_0(\gP)$ (resp.~$\widetilde{\pi}_1$).  We claim that the same holds for $\widetilde{Y}_0(\gP)$ replaced by $\widetilde{Y}_0(\gP)^\tor$ (and $\widetilde{\pi}_1$ by its extension).

To that end, first note that by Grothendieck--Serre duality, it suffices to prove the assertions for $\CO_{\widetilde{Y}_0(\gP)^{\tor}}$ (see for example the proof of \cite[Cor.~5.3.4]{KS}).  Furthermore in view of the result for the restriction to $\widetilde{Y}_0(\gP)$, it suffices to prove the assertions after localization at every point of $\widetilde{Z}$, and hence by the Theorem on Formal Functions (and faithful flatness of completion), we are reduced to proving the analogous assertions for the morphisms $\zeta:\widehat{S}_0(\gP) \to \widehat{S}$.
Applying the Theorem on Formal Functions again, we may replace $\zeta$ by the associated morphisms of toric varieties $f:T_{M,\Sigma} \to T_{M_1,\sigma}$, where $M_1 = \gd^{-1}I_1^{-1}J_1$, $M = \gd^{-1}I_1^{-1}J_2$, $\sigma$ is a cone in the chosen decomposition of $(M_1^*\otimes \R)_{\ge 0}$, and $\Sigma$ is its refinement in the chosen decomposition of $(M^*\otimes \R)_{\ge 0}$.  To prove the assertions for $f$, write it as the composite
$$T_{M,\Sigma}\stackrel{g}{\longrightarrow} T_{M,\sigma}
\stackrel{h}{\longrightarrow} T_{M_1,\sigma}.$$
By \cite[Prop.~8.5.1]{danilov} (a priori over $K$ and $k$, hence over $\CO$), we have $R^ig_*\CO_{T_{M,\Sigma}} = 0$ for $i > 0$ and $g_*\CO_{T_{M,\Sigma}} = \CO_{T_{M,\sigma}}$. Since $h$ is finite flat (of degree $[J_2:{\alpha}(J_1)]$), it follows that $R^if_*\CO_{T_{M,\Sigma}} = 0$ for $i > 0$ and $f_*\CO_{T_{M,\Sigma}} = h_*\CO_{T_{M,\sigma}}$ is locally free.

The same arguments apply to show that $R^i\widetilde{\pi}_{2,*}\CK_{\widetilde{Y}^\tor_0(\gP)/\CO}$ and $R^i\widetilde{\pi}_{2,*}\CO_{\widetilde{Y}^\tor_0(\gP)}$
vanish for $i > 0$ and are locally free (of rank $[U:U_0(\gP)] = \prod_{\gp|\gP} (p^{f_\gp} + 1)$) if $i=0$.  Recall also that analogous assertions hold for the composite 
$$\widetilde{\varphi}:\,\,\widetilde{Y}_1(\gP)\,\, \longrightarrow \,\,\widetilde{Y}_0(\gP) \,\,\stackrel{\widetilde{\pi}_1}{\longrightarrow}\,\, \widetilde{Y}$$
(but not with $\widetilde{\pi}_1$ replaced by $\widetilde{\pi}_2$, due to the assymetry in the definition of $\widetilde{Y}_1(\gP)$; see \cite[Rmk.~5.3.3]{KS}).  The extension to $\widetilde{Y}_0(\gP)^\tor$ of the result for $\widetilde{\pi}_1$, together with the isomorphism (\ref{eqn:complete}), implies that the result for $\widetilde{\varphi}$ extends also to the morphism
$\widetilde{Y}_1(\gP)^\tor \to \widetilde{Y}^\tor$ if $U = U(N)$ is $\gP$-neat.  We thus have the following extension of the results in \cite{KS}:
\begin{theorem} \label{thm:vanish}
For $j = 1,2$, the coherent sheaves 
$$R^i\widetilde{\pi}_{j,*}\CK_{\widetilde{Y}_0(\gP)^{\tor}/\CO}
\quad\mbox{and}\quad
R^i\widetilde{\pi}_{j,*}\CO_{\widetilde{Y}_0(\gP)^{\tor}}$$
vanish if $i > 0$, and are locally free of rank
$[U:U_0(\gP)] = \prod_{\gp|\gP}(p^{f_\gp} + 1)$ over $\CO_{\widetilde{Y}^\tor}$ if $i=0$.
Similarly if $U$ is $\gP$-neat, then 
$R^i\widetilde{\varphi}_{*}\CK_{\widetilde{Y}_1(\gP)^{\tor}/\CO}$
and
$R^i\widetilde{\varphi}_{*}\CO_{\widetilde{Y}_1(\gP)^{\tor}}$
vanish if $i > 0$, and are locally free of rank
$[U:U_1(\gP)] = \prod_{\gp|\gP}(p^{2f_\gp} - 1)$ if $i=0$.
\end{theorem}

\subsection{The saving trace}\label{ss:st}
We also obtain, just as in \cite[Cor.~5.3.5]{KS}, a perfect pairing
$$\widetilde{\pi}_{1,*}\CK_{\widetilde{Y}_0(\gP)^\tor/\CO}
\otimes_{\CO_{\widetilde{Y}^\tor}} \widetilde{\pi}_{1,*}\CO_{\widetilde{Y}_0(\gP)^\tor} \,\, \longrightarrow \,\,
\CK_{\widetilde{Y}^\tor/\CO}$$
extending the one over $\widetilde{Y}$, and similarly for
 $\widetilde{\pi}_1$ replaced by $\widetilde{\pi}_2$ or $\widetilde{\phi}$ (in the latter case with $\widetilde{Y}_1(\gP)^\tor$ instead of $\widetilde{Y}_0(\gP)^\tor$).  Furthermore the same argument as in the proof of \cite[Cor.~5.3.7]{KS} shows that we may replace the schemes and morphisms by their base-changes from $\CO$ to an arbitrary Noetherian $\CO$-algebra $R$.
In particular, we obtain an isomorphism
$$\widetilde{\pi}_{1,*}\CO_{\widetilde{Y}^\tor_0(\gP)_R}
\,\,\stackrel{\sim}{\longrightarrow}\,\,
\HOM_{\CO_{\widetilde{Y}_R^\tor}}(\widetilde{\pi}_{1,*}\CK_{\widetilde{Y}^\tor_0(\gP)_R/R},\CK_{\widetilde{Y}^\tor_R/R})$$
extending the one over $\widetilde{Y}_R$, the image of the unit section being the trace morphism
$$\widetilde{\pi}_{1,*}\CK_{\widetilde{Y}^\tor_0(\gP)_R/R}\,\,\longrightarrow \,\,\CK_{\widetilde{Y}^\tor_R/R}.$$

As in \cite[\S5.4]{KS}, we may combine the trace morphism with the Kodaira--Spencer isomorphisms over $\widetilde{Y}_0(\gP)^\tor$ and $\widetilde{Y}^\tor$ and the extension to $\widetilde{Y}_0(\gP)^\tor$ of the canonical isomorphism $\widetilde{\pi}_2^*\widetilde{\delta} \stackrel{\sim}{\longrightarrow} \widetilde{\pi}_1^*\widetilde{\delta}$ (defined by multiplication by $\Nm(\gP)^{-1}$ to obtain a morphism
$$\widetilde{\pi}_{1,*}(\CI_{\widetilde{Z}_0(\gP)}\otimes_{\CO_{\widetilde{Y}_0(\gP)^\tor}} \widetilde{\pi}_2^*\widetilde{\omega}) \,\,\longrightarrow\,\, 
\CI_{\widetilde{Z}} \otimes_{\widetilde{Y}^\tor} \widetilde{\omega}$$ over $\widetilde{Y}^\tor$
(where $\CI_{\widetilde{Z}_0(\gP)}$ and $\CI_{\widetilde{Z}}$ denote the ideal sheaves defining the cuspidal divisors).
The morphism extends the pull-back to $\widetilde{Y}$ of the saving trace defined in \cite[(51)]{KS} (in the case $\gP = \gp$), as does the morphism
$$\widetilde{\st}: \widetilde{\pi}_{1,*}\widetilde{\pi}_2^*\widetilde{\omega} \,\,\longrightarrow\,\, 
\widetilde{\omega}$$
obtained from its composite with the morphism
$$\CI_{\widetilde{Z}} \otimes_{\CO_{\widetilde{Y}^\tor}}
\widetilde{\pi}_{1,*}\widetilde{\pi}_2^*\widetilde{\omega}
\,\,= \,\, 
\widetilde{\pi}_{1,*}
(\widetilde{\pi}_1^* \CI_{\widetilde{Z}} \otimes_{\CO_{\widetilde{Y}_0(\gP)^\tor}}\widetilde{\pi}_2^*\widetilde{\omega})\,\,\longrightarrow\,\,
\widetilde{\pi}_{1,*}(\CI_{\widetilde{Z}_0(\gp)}\otimes_{\CO_{\widetilde{Y}_0(\gP)^\tor}} \widetilde{\pi}_2^*\widetilde{\omega})$$
induced by $\widetilde{\pi}_1^*\CI_{\widetilde{Z}} \to \CI_{\widetilde{Z}_0(\gP)}$.

For an arbitrary Noetherian $\CO$-algebra $R$, we have the flatness (resp.~vanishing) of
$R^i\widetilde{\pi}_{1,*}\widetilde{\pi}_2^*\widetilde{\omega}$ over $\widetilde{Y}^\tor$ for $i=0$ (resp.~$i > 0$) as a consequence of the analogous assertions with $\widetilde{\omega}$ replaced by $\CO_{\widetilde{Y}^\tor}$ and the canonical trivialization of $\xi^*\widetilde{\omega}$.  It follows that 
$(\widetilde{\pi}_{1,*}\widetilde{\pi}_2^*\widetilde{\omega})_R = \widetilde{\pi}_{1,*}(\widetilde{\pi}_2^*\widetilde{\omega}_R)$, yielding a saving trace
$$\widetilde{\st}_R:\,\,\widetilde{\pi}_{1,*}(\widetilde{\pi}_2^*\widetilde{\omega}_R)\,\,\longrightarrow\,\, \widetilde{\omega}_R$$
over $\widetilde{Y}^\tor_R$ (as usual omitting the subscripts for base-changes of morphisms).

Finally we describe the effect of the saving trace on completions along $\widetilde{Z}$.  To that end, let $[\utH_1]$ be a cusp in $\widetilde{C}$ and $\xi:\widehat{S} \to \widetilde{Y}^\tor$ the formal scheme as in (\ref{eqn:canonical}), so that
$\xi^*\widetilde{\omega} \cong \wedge^d(I_1^{-1})\otimes \CO_{\widehat{S}}$.
Recall that the cusps $[\utH_1,\utH_2,\alpha]$ in $\widetilde{C}_0(\gp)$ lying over $[\utH_1]$ are in bijection with factorizations $\gP = \gq_1\gq_2$,
where we define $\gq_1$ and $\gq_2$ by
$\alpha(I_1) = \gq_2 I_2$ and $\alpha(J_1) = \gq_1J_2$,
Thus $\xi^*(\widetilde{\pi}_{1,*}\widetilde{\pi}_2^*\widetilde{\omega})$ is comprised of $2^t$ summands of the form 
$$\zeta_*(\xi_0^*(\widetilde{\pi}_{1,*}\widetilde{\pi}_2^*\widetilde{\omega})) \cong \bigwedge\nolimits^d(I_2^{-1})\otimes \zeta_*\CO_{\widehat{S}_0(\gp)},$$
where
$$ \xi_0: \widehat{S}_0(\gp) \to \widetilde{Y}_0(\gp)^\tor\quad\mbox{and}\quad \zeta: \widehat{S}_0(\gp) \to \widehat{S}$$
are as above and $t = \#\{\gp|\gP\}$.  It follows from \cite[Prop.~3.2.2]{KS} that the saving trace pulls back via $\xi$ to the sum of the morphisms
corresponding under the above isomorphisms to the maps
$$\beta \otimes \Nm(\gq_1)^{-1}\tr:\,\,
\bigwedge\nolimits^d(I_2^{-1})\otimes \zeta_*\CO_{\widehat{S}_0(\gp)} \,\,\longrightarrow\,\,
\bigwedge\nolimits^d(I_1^{-1})\otimes \CO_{\widehat{S}},$$
where $\tr$ is the trace relative to the finite flat morphism $\zeta_*\CO_{\widehat{S}_0(\gP)} \to \CO_{\widehat{S}}$ and $\beta = \Nm(\gq_2)^{-1}\wedge^d\alpha^*$ (writing $\alpha^*:I_2^{-1} \to I_1^{-1}$ for the map induced by $\alpha$).  Note that $\beta$ is an isomorphism and the second factor may be described  locally on $\widehat{S}$ as $\Nm(\gq_1)^{-1}$ times the completion of the trace of the finite flat morphism $h:T_{M,\sigma} \to T_{M_1,\sigma}$.  In particular on global sections the resulting map is given by
$$\begin{array}{ccc}
\Gamma(\widehat{S}_0(\gP),\CO_{\widehat{S}_0(\gP)})
&\to&
\Gamma(\widehat{S},\CO_{\widehat{S}}) \\
\wr||&&\wr|| \\
\CO[[q^m]]_{m \in N^{-1}M_{+} \cup\{0\} } & \to &
\CO[[q^m]]_{m \in N^{-1}M_{1,+} \cup\{0\} } \\ && \\
\sum r_mq^m  & \mapsto &
\sum r_{\alpha(m)}q^m\end{array}$$
(writing $\alpha$ also for the induced map 
$N^{-1}M_{1} \to N^{-1}M$).
Finally it follows that the base-change $\widetilde{\st}_R$ is given by the same formula with $\CO$ replaced by $R$.
We conclude:
\begin{proposition} \label{prop:stq}
The pull-back to $\widetilde{Y}_R$ of the saving trace (i.e., \cite[(51)]{KS}) extends to a morphism $\widetilde{\st}_R: \widetilde{\pi}_{1,*}(\widetilde{\pi}_2^*\widetilde{\omega}_R) \to \widetilde{\omega}_R$ over $\widetilde{Y}_R^\tor$ whose completion along the complement of $\widetilde{Y}_R$ is the morphism whose effect on global sections is induced by the maps
$$\begin{array}{ccc}
\bigwedge\nolimits^d(I_2^{-1}) \otimes_{\CO}
R[[q^m]]_{m \in N^{-1}M_{+} \cup\{0\} } & \to &
\bigwedge\nolimits^d(I_1^{-1}) \otimes_{\CO}
R[[q^m]]_{m \in N^{-1}M_{1,+} \cup\{0\} } \\ && \\
b \otimes \sum r_mq^m  & \mapsto &
\beta(b)\otimes \sum r_{\alpha(m)}q^m.\end{array}$$
\end{proposition}

\section{Hecke operators on $q$-expansions} \label{sec:Tp}
Recall that in \cite{theta}, we computed the effect on $q$-expansions of all the weight-shifting operators defined there, namely partial $\Theta$ and Frobenius operators.\footnote{and for completeness, the simpler and well-known effect of multiplication by partial Hasse invariants} We do the same here for Hecke operators.  

For Hecke operators associated to primes not dividing $p$, it is well-known that the same computation as for classical Hilbert modular forms, i.e. the case $R = \C$, carries over to arbitrary bases.  This is explained in \cite{DS1} in the case that $p$ is unramified in $F$, and we do the same here in general for completeness.  

Hecke operators at primes dividing $p$ are also considered in \cite{DS1}, but in addition to the assumption that $p$ be unramified in $F$, constraints are placed on the weight to allow for a simpler more ad hoc definition of the operator $T_\gp$.  A more general, indeed optimal, construction of $T_\gp$ is given in \cite{KS} (see also \cite{FP} for $p$ unramified in $F$ and \cite{ERX} for work in this direction allowing ramfication at $p$).  The effect of $T_\gp$ on $q$-expansions is also considered in \cite{DW} and \cite{DMPhD}, premised on partial results in \cite{ERX}; we complete the analysis here using the operators defined in \cite{KS}.

\subsection{Hilbert modular forms} \label{ss:hmfs}
For $\vec{k}, \vec{m} \in \Z^\Theta$, a Noetherian $\CO$-algebra $R$ and (sufficiently small) $U$ such that $\chi_{\vec{k}+2\vec{m},R} = 1$ on $\CO_F^\times \cap U$, we define the space of {\em Hilbert modular forms of weight $(\vec{k},\vec{m})$ and level $U$ over $R$} to be
$$M_{\vec{k},\vec{m}}(U;R) := H^0(Y_R,\CA_{\vec{k},\vec{m},R})
   = H^0(Y_R^{\min},j_* \CA_{\vec{k},\vec{m},R}).$$
Note that the hypotheses imply that either
\begin{itemize}
\item $p^nR = 0$ for some $n > 0$, or
\item $k_\theta + 2m_\theta$ is independent of $\theta$.
\end{itemize}
{\em We assume henceforth that one of these holds.}  Note that in either case $M_{\vec{k},\vec{m}}(U;R)$ is defined for all sufficiently small $U$ (containing $\GL_2(\CO_{F,p})$), so we may define
$$M_{\vec{k},\vec{m},R} : = \lim\limits_{\longrightarrow_U}
  M_{\vec{k},\vec{m}}(U;R),$$
where the limit is taken over all such $U$ with respect to the (injective) morphisms $M_{\vec{k},\vec{m}}(U;R) \to M_{\vec{k},\vec{m}}(U';R)$ for $U' \subset U$.

Recall that for $g \in \GL_2(\A_{F,\f}^{(p)})$ and $U' \subset gUg^{-1}$ (where $U$ and $U'$ are sufficiently small), we have the morphisms $M_{\vec{k},\vec{m}}(U;R) \to M_{\vec{k},\vec{m}}(U';R)$ defined as $\| \det(g) \|$ times the composite
$$H^0(Y_R,\CA_{\vec{k},\vec{m},R}) \longrightarrow
 H^0(Y'_R,\rho_g^*\CA_{\vec{k},\vec{m},R}) \longrightarrow
 H^0(Y'_R,\CA'_{\vec{k},\vec{m},R}),$$
the transition maps in the definition of $M_{\vec{k},\vec{m},R}$
being the special case where $g = 1$.
These satisfy the usual compatibilities and hence define an action of $\GL_2(\A_{F,\f}^{(p)})$ on $M_{\vec{k},\vec{m},R}$.
Furthermore since $M_{\vec{k},\vec{m}}(U';R)^{U/U'}$ coincides with $M_{\vec{k},\vec{m}}(U;R)$ for any normal subgroup $U'$ of $U$ (whenever both have already been defined), we may extend the definition to arbitrary open compact $U = U^p\GL_2(\CO_{F,p})$ by letting $M_{\vec{k},\vec{m}}(U;R) = M_{\vec{k},\vec{m},R}^{U^p}$.
Note also that the definition of $M_{\vec{k},\vec{m},R}$ is functorial in $R$, so that $\CO$-algebra morphisms $R \to R'$ give rise to $\GL_2(\A_{F,\f}^{(p)})$-equivariant $R'$-linear morphisms
$M_{\vec{k},\vec{m},R} \otimes_R R' \to M_{\vec{k},\vec{m},R'}$; moreover this is an isomorphism whenever $R'$ is flat over $R$, or indeed if both are flat over $\CO$.

We remark also that obvious variants of the above construction yield analogous assertions with $M_{\vec{k},\vec{m}}(U;R)$ replaced by $M^{(i)}_{\vec{k},\vec{m}}(U_0(\gP);R) := 
H^0(Y_0(\gP)_R,\CA_{\vec{k},\vec{m},R}^{(i)})$ (for $i=1,2$) or
$M^{\flat}_{\vec{k},\vec{m}}(U_1(\gP);R) := 
H^0(Y_1(\gP)_R,\CA_{\vec{k},\vec{m},R}^{\flat})$, and we denote the resulting limits $M_0(\gP)^{(i)}_{\vec{k},\vec{m},R}$ and $M_1(\gP)^{\flat}_{\vec{k},\vec{m},R}$.  Note also that we have $\GL_2(A_{F,\f}^{(p)})$-equivariant injections
$M_{\vec{k},\vec{m},R} \hookrightarrow 
M_0(\gP)^{(i)}_{\vec{k},\vec{m},R}$ induced by the morphisms $\pi_i$,
and a natural $\GL_2(A_{F,\f}^{(p)})$-equivariant 
action of $(\CO/\gP)^\times$
on $M_1(\gP)^{\flat}_{\vec{k},\vec{m},R}$ under which the
invariants may be identified with 
$M_0(\gP)^{(1)}_{\vec{k},\vec{m},R}$.

\subsection{The $q$-expansion Principle} \label{ss:qep}
For each cusp $c = [H,I,[\lambda],[\eta]] \in C = C_U$, we have the $q$-expansion map
$$\bq_c:  M_{\vec{k},\vec{m}}(U;R)  \to (j_*\CA_{\vec{k},\vec{m},R})^{\wedge} \hookrightarrow
D_{\vec{k},\vec{m}} \otimes_\CO R[[q^m]]_{m \in N^{-1}M_+ \cup \{0\}}, $$ where the completion is along $Z_{c,R}$ and the inclusion depends on a choice of splitting $H \isoto J \times I$.  The notation here is as in \S\ref{ss:KP}, so in particular 
$M = \gd^{-1}I^{-1}J$ and $D_{\vec{k},\vec{m}} = 
(I^{-1})_\theta^{\otimes k_\theta} 
   \otimes (\gd(IJ)^{-1})_\theta^{\otimes m_\theta}$ depend on $c$, and $N$ is such that $U(N) \subset U$ and $\zeta_N \in \CO$.
The above definition of $\bq_c$ assumes a priori that $U$ is sufficiently small, but we may drop this assumption by letting $\bq_c(f) = \bq_{c'}(f)$ for any $c' \in C_{U(N)}$ in the preimage of $c$ under the natural projection.  The resulting $q$-expansion is independent of the choice of such a $c'$ (as a special case of 
Proposition~\ref{prop:koecher}(3)); furthermore the definition of $\bq_c$ is independent (in the obvious sense) of the choice of $N$ such that $U(N) \subset U$.

More generally for any $S \subset C$, we define
$$\bq_S : M_{\vec{k},\vec{m}}(U;R) \,\, \longrightarrow \,\, \bigoplus_{c \in S} \left( D_{\vec{k},\vec{m}} \otimes_\CO R[[q^m]]_{m \in N^{-1}M_+ \cup \{0\}} \right) $$
as the direct product of the maps $\bq_c$.
We then have the following $q$-expansion Principle, which can be proved exactly as in \cite[Thm.~6.7]{rap} if $U$ is sufficiently small; the assumption can then removed by applying the result to sufficiently small $U'$ normal in $U$ and taking invariants under $U/U'$.
\begin{proposition} \label{prop:qexp}
Suppose that $S$ meets every connected component of $Y^{\min}$, or  equivalently, the restriction to $S$ of the map
$$B(F)_+ \backslash \GL_2(\A_{F,\f}) / U \,\, \stackrel{\det}{\longrightarrow} \,\,
F^\times_+ \backslash \A_{F,\f}^\times / \det(U)$$
is surjective.
Then
\begin{enumerate}
\item $\bq_S$ is injective;
\item if $R' \subset R$, $f \in M_{\vec{k},\vec{m}}(U;R)$ and 
$$\bq_S(f)\,\, \in \,\, \bigoplus_{c \in S} \left( D_{\vec{k},\vec{m}} \otimes_\CO R'[[q^m]]_{m \in N^{-1}M_+ \cup \{0\}} \right),$$
then $f \in M_{\vec{k},\vec{m}}(U;R')$.
\end{enumerate}
\end{proposition}

More generally for any cusp $c \in C_0(\gP)$, we may define $q$-expansion maps
$$\bq^{(i)}_c:  M^{(i)}_{\vec{k},\vec{m}}(U_0(\gP);R)  \to
D^{(i)}_{\vec{k},\vec{m}} \otimes_\CO R[[q^m]]_{m \in N^{-1}M_+ \cup \{0\}}$$
for $i=1,2$, where now $M = \gd^{-1}I_1^{-1}J_2$, and more generally $\bq^{(i)}_S$ for any $S \subset C_0(\gP)$.  
However since irreducible components of $Y_0(\gP)_R^{\min}$ need not contain cusps, the analogue of Proposition~\ref{prop:qexp} fails.

Recall also from Proposition~\ref{prop:Y1pcomp}(1) that connected components of the complement of $Y_1(\gP)$ in $Y_1(\gP)^{\min}$ correspond to clasps $c\ell = [\underline{H},\Xi] \in C_{1/2}(\gP)$, and by Proposition~\ref{prop:Y1pqexp}, we have $q$-expansion maps
$$\bq_{c\ell}:  M^{\flat}_{\vec{k},\vec{m}}(U_1(\gP);R)  \to
D_{\vec{k},\vec{m}}\otimes_\CO T_{\gor} \otimes_\CO R[[q^m]]_{m \in (\gq^{-1}N^{-1}M)_+ \cup \{0\}},$$
where $M = \gd^{-1}I^{-1}J$, $\gq = \Ann_{\CO_F}(\Xi)$,
$\gor = \gq^{-1}\gP$ and $T_{\gor}$ is the finite flat $\CO$-algebra representing generators of $\mu_p \otimes I/\gor I$.

The inclusions
$M_{\vec{k},\vec{m}}(U;R) \stackrel{\pi_i^*}{\longrightarrow} M^{(i)}_{\vec{k},\vec{m}}(U_0(\gP);R)$
have the obvious effect on $q$-expansions.  More precisely
$\bq_c^{(i)}(\pi_i^*(f))$ is the image of $\bq_{\pi_i^\infty(c)}(f)$ under the identification of $D_{\vec{k},\vec{m}}^{(i)}$ (for $c = [\underline{H}_1,\underline{H}_2,\alpha]$) with $D_{\vec{k},\vec{m}}$ (for $\pi_i^\infty(c) = [\underline{H}_i]$) and the inclusion of power series rings obtained from the maps
$I_1 \hookrightarrow I_i$ and $J_i \hookrightarrow J_2$
(one of which is the identity and the other is induced by $\alpha$; see Theorem~\ref{thm:Y0min}(3)).  

We may similarly describe the effect of
$M^{(1)}_{\vec{k},\vec{m}}(U_0(\gP);R)\stackrel{h^*}{\longrightarrow} M^{\flat}_{\vec{k},\vec{m}}(U_1(\gP);R)$
on $q$-expansions.  Indeed recall that
$c\ell = [\underline{H},\Xi] \in C_{1/2}(\gP)$ maps to $c =[\underline{H}_1,\underline{H}_2,\alpha]$, where $\underline{H}_1 = \underline{H}$ and $\gq^{-1}J \isoto J_2$ (via $\alpha$).
We may thus identify $D_{\vec{k},\vec{m}}^{(i)}$ (resp.~$M$)
for $c$ with $D_{\vec{k},\vec{m}}$ (resp.~$\gq^{-1}M$)
for $c\ell$, and $\bq_{c\ell}(h^*(f))$ is obtained from $\bq_{c}(f)$ by tensoring with the structure map $\CO \to T_\gor$.

The effect on $q$-expansions of the action of $d \in (\CO/\gP)^\times$ is given simply by 
$$\bq_{[\underline{H},\Xi]}(d^*f) = d_{\gor}^*(\bq_{[\underline{H},d\Xi]}(f)),$$ where $d_{\gor}^*$ is defined
by the action of $(\CO/\gP)^\times$ on $T_\gor$; thus if $t = \sum t^\xi$ (in the notation of Proposition~\ref{prop:Y1pcomp}(2)),
then $d_{\gor}^*t = \sum \xi(d)t^\xi$.

\subsection{Cusp forms} \label{ss:cuspforms}
For $c \in C$, we let
$$\e_c: M_{\vec{k},\vec{m}}(U;R) \,\,\longrightarrow \,\, D_{\vec{k},\vec{m}} \otimes_{\CO} R$$
denote the $R$-linear (evaluation) map sending $f$ to the constant coefficient
of $\bq_c(f)$, and we similarly define
$$\e_S: M_{\vec{k},\vec{m}}(U;R) \,\,\longrightarrow \,\,\bigoplus_{c\in S}( D_{\vec{k},\vec{m}} \otimes_{\CO} R )$$
for any $S \subset C$.  We define the space of
{\em cuspidal Hilbert modular forms of weight $(\vec{k},\vec{m})$ and level $U$ over $R$} to be
$S_{\vec{k},\vec{m}}(U;R) : = \ker(\e_C)$.
Thus $f \in M_{\vec{k},\vec{m}}(U;R)$ is cuspidal if
$\e_c(f) = 0$ for all $c \in C$, in which case we
also refer to $f$ as a {\em cusp form} (of 
weight $(\vec{k},\vec{m})$ and level $U$ over $R$).

Note that $\e_c$ is independent of the choice of splittings in the definition of $\bq_c$, and that it takes values in
$$(D_{\vec{k},\vec{m}} \otimes_{\CO} R)^{\Gamma_{\calC,U}}
 = D_{\vec{k},\vec{m}} \otimes_{\CO} \ga,$$
where $\ga = R[b]$ and $b \in \CO$ generates the ideal
$$ \left\langle \, \chi_{\vec{m}}(\alpha) \chi_{\vec{k}+\vec{m}}(\delta) - 1 \,\left|\,\smat{\alpha}{\beta}{0}{\delta} \in \Gamma_{\calC,U}\,\right.\right\rangle.$$
In particular, if $R$ is flat over $\CO$, then $\ga = 0$ for all cusps $c \in C$, and hence $S_{\vec{k},\vec{m}}(U;R) = M_{\vec{k},\vec{m}}(U;R)$, unless $\vec{k}$ and $\vec{m}$ are parallel in the sense that the pair $(k_\theta,m_\theta)$ is independent of $\theta$.  On the other hand if 
$(k_\theta,m_\theta)$ is independent of $\theta$, or if $p^nR = 0$
for some $n > 0$, then for sufficiently small $U$, we have $\ga = R$ for all $c \in C$.

Note that $f \in M_{\vec{k},\vec{m}}(U;R)$ is cuspidal if and only if its image in $M_{\vec{k},\vec{m}}(U';R)$ for some (hence all) $U' \subset U$.  Furthermore 
Proposition~\ref{prop:koecher}(3) (in the case $\gP = \CO_F$)
implies that the subspace
$$S_{\vec{k},\vec{m},R} := \lim\limits_{\longrightarrow_U} S_{\vec{k},\vec{m}}(U;R)
\subset M_{\vec{k},\vec{m},R} $$
is stable under the action of $\GL_2(\A_{F,\f}^{(p)})$, and
that $S_{\vec{k},\vec{m}}(U;R) = S_{\vec{k},\vec{m},R}^{U^p}$
for all $U = U^p\GL_2(\CO_{F,p})$.   Note also that the definition of $S_{\vec{k},\vec{m},R}$ is functorial in $R$, and that
$S_{\vec{k},\vec{m},R'} = S_{\vec{k},\vec{m},R} \otimes_R R'$ if $R'$ is flat over $R$ (or both are flat over $\CO$).

More generally for $c \in C_0(\gP)$, we have the evaluation maps
$$\e_c^{(i)}: M_{\vec{k},\vec{m}}^{(i)}(U_0(\gP);R)
  \longrightarrow D^{(i)}_{\vec{k},\vec{m}} \otimes_\CO R,$$
and we similarly define $\e_S^{(i)}$ for $S \subset C_0(\gP)$, and let $S_{\vec{k},\vec{m}}^{(i)}(U_0(\gP);R) = \ker(\e_{C_0(\gP)}^{(i)})$.  For $c\ell \in C_{1/2}(\gP)$ the evaluation maps take the form
$$\e_{c\ell}: M_{\vec{k},\vec{m}}^{\flat}(U_1(\gP);R)
  \longrightarrow D_{\vec{k},\vec{m}} \otimes_\CO T_\gor \otimes_\CO R$$
and $\e_S$ for $S \subset C_{1/2}(\gP)$, and we let
$S_{\vec{k},\vec{m}}^{\flat}(U_1(\gP);R) = \ker(\e_{C_{1/2}(\gP)})$.  The same considerations as above yield $\GL_2(\A_{F,\f}^{(p)})$-submodules
$$S_0(\gP)_{\vec{k},\vec{m},R}^{(i)} \subset M_0(\gP)_{\vec{k},\vec{m},R}^{(i)} 
 \quad\mbox{and} 
\quad S_1(\gP)_{\vec{k},\vec{m},R}^{\flat}\subset
M_1(\gP)_{\vec{k},\vec{m},R}^{\flat} 
 $$
for which assertions analogous to those for $S_{\vec{k},\vec{m},R}$ hold.  Furthermore $S_{\vec{k},\vec{m},R}$ is the preimage of $S_0(\gP)_{\vec{k},\vec{m},R}^{(i)}$ under the injection $M_{\vec{k},\vec{m},R} \hookrightarrow M_0(\gP)_{\vec{k},\vec{m},R}^{(i)}$, and the action of $(\CO/\gP)^\times$ on 
$M_1(\gP)_{\vec{k},\vec{m},R}^{\flat}$ restricts to one on
$S_1(\gP)_{\vec{k},\vec{m},R}^{\flat}$ with invariants
$S_0(\gP)_{\vec{k},\vec{m},R}^{(1)}$.

\subsection{Cusps at $\infty$} \label{ss:cuspsinf}
We call $c \in C = C_U$ a {\em cusp at $\infty$} if it is of the form
$$c_t := B(\CO_{F,(p)})_+\smat{t}{0}{0}{1}U^p $$
for some $t \in (\AA_{F,\f}^{(p)})^\times$.
Thus the data $\underline{H} = (H,I,[\lambda],[\eta])$ corresponding to $c_t$ is given by 
\begin{itemize}
\item $H = J \times I$, where $I = \CO_F$ and $J = J_t = t^{-1}\widehat{\CO}_F \cap F$,
\item $\lambda$ is the identification $J_{(p)} = \CO_{F,(p)}$
  (or equivalently, multiplication by any element of $\CO_{F,(p),+}^\times$),
\item $\eta$ is defined by $\eta(x,y) = (t^{-1}x,y)$.
\end{itemize}
Note that Proposition~\ref{prop:qexp} holds with $S$ as the set of cusps at $\infty$.

For simplicity, we assume throughout this section $U = U_1(\gn)$ or $U(\gn)$ for some $\gn$ prime to $p$.
We let $V_{\gn} = \ker((\widehat{\CO}_{F}^{(p)})^\times \to (\CO_F/\gn)^\times)$ and 
$$E_\gn = \CO_F^\times \cap V_{\gn}
 = \{\,\mu\in \CO_F^\times\,|\, \mu \equiv 1\bmod \gn\,\}.$$
Note that in the case of $U_1(\gn)$, we have $c_t = c_{t'}$ if and 
only if $t$ and $t'$ define the same strict ideal class of $F$.
Similarly in the case of $U(\gn)$, the cusps at $\infty$ are indexed by the strict ray class group of conductor $\gn$.  Note that in either case, there is a unique cusp at $\infty$ on each connected component of $Y^{\min}$.

Recall that $R$ is a Noetherian $\CO$-algebra and $\vec{k},\vec{m} \in \Z^\Theta$ are such that either $p^nR = 0$ for some $n >0$,
or $k_\theta + 2m_\theta$ is independent of $\theta$. 
If  $U= U_1(\gn)$ and $f \in M_{\vec{k},\vec{m}}(U;R)$, then
Proposition~\ref{prop:koecher}(2) implies
(for sufficiently small $\gn$, and hence for any $\gn$) 
that the $q$-expansion\footnote{We always implicitly choose the splitting defined by the equality $H = J \times I$}
$\bq_{c_t}(f)$ takes values in the $P_t$-module $Q_t = $
$$\left \{\, \left. \sum_{m \in (\gd^{-1}J)_+ \cup \{0\}}\!\!\!\! \!\!\!\!\!\!  b \otimes r_mq^m \,
\right|\,
\mbox{$r_m = \chi_{\vec{m},R}(\nu) r_{\nu m}$ for all
$\nu \in \CO_{F,+}^\times$, $m  \in (\gd^{-1}J)_+ \cup \{0\}$}\,\right\},$$
where 
$P_t = R[[q^m]]_{m \in (\gd^{-1}J)_+\cup \{0\}}^{\CO_{F,+}^\times}$ and $b = b_t$ is any basis for 
$$D_{\vec{k},\vec{m},t} = \bigotimes_{\theta \in \Theta} (\gd J_t^{-1} \otimes_{\CO_F,\theta} \CO)^{\otimes m_\theta}.$$
Similarly if $U = U(\gn)$, then the same assertions hold with $J$ replaced by $\gn^{-1}J$ and $\CO_{F,+}^\times$ by $E_{\gn,+}$ in the definitions of $P_t$ and $Q_t$ (but the same $D_{\vec{k},\vec{m},t}$).

Note that if $U = U_1(\gn)$ (resp.~$U(\gn)$), then $c_t = c_{t'}$
if and only if $t' = \alpha t u$ for some $\alpha \in \CO_{F,(p),+}^\times$ and $u \in (\widehat{\CO}_F^{(p)})^\times$ (resp.~$V_\gn$).
The resulting isomorphism $P_t \isoto P_{t'}$ defined by $q^m \mapsto q^{\alpha m}$ is independent of the choice of such an $\alpha$. Similarly so is the resulting isomorphism $Q_t \isoto Q_{t'}$, which is given by the map defined by $q^m \mapsto q^{\alpha m}$ and the isomorphism $D_{\vec{k},\vec{m},t} \isoto D_{\vec{k},\vec{m},t'}$ 
induced by multiplication by $\alpha^{-1}$ on $J_t^{-1}$.
For $f \in M_{\vec{k},\vec{m}}(U;R)$ and $m \in (\gd^{-1}J_t)_+ \cup \{0\}$ (resp.~ $(\gd^{-1}\gn^{-1}J_t)_+ \cup \{0\}$), we write $r_m^t(f)$ for the coefficient of $q^m$ in $\bq_{c_t}(f)$, viewed as an element of $D_{\vec{k},\vec{m},R} := D_{\vec{k},\vec{m},1}\otimes_{\CO} R$ via the isomorphism $D_{\vec{k},\vec{m},t} \isoto D_{\vec{k},\vec{m},1}$ induced by the identification $J_{(p)} \cong \CO_{F,(p)}$.  Note that
if $t' = \alpha t u$ for some $\alpha$ and $u$ as above, then
$J_{t'} = \alpha^{-1}J_t$ and
\begin{equation}\label{eqn:rmt} r_m^{t'}(f) = \chi_{\vec{m},R}(\alpha) r_{\alpha m}^{t}(f)\end{equation}
for all $m \in (\gd^{-1}J_{t'})_+\cup\{0\}$
(resp.~$(\gd^{-1}\gn^{-1}J_{t'})_+\cup\{0\}$).

\subsection{Hecke operators outside $p$} \label{ss:Tv}
We continue to assume that $U$ is of the form $U_1(\gn)$ or $U(\gn)$ (for some $\gn$ prime to $p$).
Recall from \S\ref{ss:hmfs} that $M_{\vec{k},\vec{m},R}$ is equipped with an action of $\GL_2(\AA_{F,\f}^{(p)})$ such that
$M_{\vec{k},\vec{m}}(U;R) = M_{\vec{k},\vec{m},R}^U$.
We may therefore define commuting $R$-linear Hecke operators $T_v$
on $M_{\vec{k},\vec{m}}(U;R)$,
for all primes $v\nmid p$ (resp.~$v\nmid\gn p$) of $\CO_F$
if $U = U_1(\gn)$ (resp.~$U(\gn)$),
by the action of the double cosets
$$T_v = \left[ U\smat{\varpi_v}{0}{0}{1} U\right],$$
where $\varpi_v$ is any uniformizer of $\CO_{F,v}$ (viewed as an element of $F_v^\times \subset (\A_{F,\f}^{(p)})^\times$).
Furthermore it follows from the stability of $S_{\vec{k},\vec{m},R}$ under the action of $\GL_2(\A_{F,\f}^{(p)})$ 
(see \S\ref{ss:cuspforms}) that
the operators $T_v$ restrict to endomorphisms of 
$S_{\vec{k},\vec{m}}(U;R) =  S_{\vec{k},\vec{m},R}^U$.

Similarly we have the operators 
$$S_w = \left[ U\smat{\varpi_w}{0}{0}{\varpi_w} U \right]$$
on $M_{\vec{k},\vec{m}}(U;R)$ for all $w\nmid\gn p$, preserving $S_{\vec{k},\vec{m}}(U;R)$ and commuting with each other and the operators $T_v$.
Note that restricting the action of $\GL_2(\A_{F,\f}^{(p)})$ to its center defines an action of $(\A_{F,\f}^{(p)})^\times$ on $M_{\vec{k},\vec{m}}(U;R)$ whose restriction to $V_\gn$ (resp.~$\CO_{F,(p),+}^\times$) is trivial (resp.~$\chi_{\vec{k}+2\vec{m} - \vec{2},R}$),  where $V_\gn: = \ker((\widehat{\CO}_{F}^{(p)})^\times \to (\CO_F/\gn)^\times)$.  The action of the operator $S_v$ is simply that of $\varpi_v$ under this identification.

The effect of the operators $T_v$ on $q$-expansions at cusps at $\infty$ is then given as follows:
\begin{proposition} \label{prop:Tvonq} Suppose that $U = U_1(\gn)$ (resp.~$U(\gn)$) for some $\gn$ prime to $p$, and let $f \in M_{\vec{k},\vec{m}}(U;R)$.  
\begin{enumerate}
\item If $v$ is a prime of $\CO_F$
not dividing $\gn p$, then 
$$r_m^t(T_v f) = r_m^{\varpi_vt}(f) +  \Nm_{F/\Q}(v)r_m^{\varpi_v^{-1}t}(S_vf)$$
for all $t \in (\AA_{F,\f}^{(p)})^\times$ and 
$m \in (\gd^{-1}J_t)_+ \cup \{0\}$
(resp.~ $(\gd^{-1}\gn^{-1}J_t)_+ \cup \{0\}$).
\item If $U = U_1(\gn)$ and $v$ is a prime of $\CO_F$
dividing $\gn$, then 
$$r_m^t(T_v f) = r_m^{\varpi_vt}(f)$$
for all $t \in (\AA_{F,\f}^{(p)})^\times$ and 
$m \in (\gd^{-1}J_t)_+ \cup \{0\}$.
\end{enumerate}
\end{proposition}
\begin{proof} Suppose first that $U = U(\gn)$.
To prove (1), we may replace $\gn$ by $\gm\gn$ for
any $\gm$ prime to $vp$ and hence assume $\gn$
is sufficiently small.  
By a standard double coset computation, we have
$T_v f = \sum g_\iota f$ where the sum is over $\iota \in
\PP^1(\CO_F/v)$, $g_\infty = \smat{0}{1}{\varpi_v}{0} \in \GL_2(F_v) \subset \GL_2(\A_{F,\f}^{(p)})$ and
$g_\iota = \smat{\varpi_v}{\widetilde{\iota}}{0}{1}$
for any lift $\widetilde{\iota} \in \CO_{F,v}$ of
$\iota \in \CO_F/v$.
We may therefore compute the $q$-expansion of $T_v(f)$ at $c_t$ by 
summing those of the $g_\iota(f) \in M_{\vec{k},\vec{m}}(U';R)$ at the cusp $c_t' \in C_{U'}$, where $U' = U(\gn v)$.

We now apply Proposition~\ref{prop:koecher}(3) (with $\gP = \CO_F$) to determine the effect of the map $g_{\iota}: M_{\vec{k},\vec{m}}(U;R)\to M_{\vec{k},\vec{m}}(U';R)$ on $q$-expansions.  
First note that under $\rho_{g_\iota}:Y'^{\min} \to Y^{\min}$, we have $\rho_{g_\infty}(c_t') = \rho_h(c_{\varpi_v^{-1} t})$ and
$\rho_{g_{\iota}}(c_t') = c_{\varpi_v t}$ otherwise,
where $\rho_h$ is the automorphism of $Y^{\min}$ induced by
$h= \smat{\varpi_v}{0}{0}{\varpi_v}$.  More precisely, we have
$$\smat{t}{0}{0}{1}g_{\infty}\in
\smat{\varpi_v^{-1} t}{0}{0}{1}hU
\quad\mbox{and}\quad\smat{t}{0}{0}{1}g_{\iota}\in
\smat{1}{\epsilon_\iota}{0}{1}
\smat{\varpi_v t}{0}{0}{1}U$$
for $\iota \in \CO_F/v$, where $\epsilon_\iota \in \gn J^{-1}$ is any lift of the class of $t_v\widetilde{\iota}$ in $t_v\CO_{F,v}/vt_v\CO_{F,v} \cong \gn J^{-1}/v\gn J^{-1}$.  Taking into account the factors of $||\det(g_\iota)|| = \Nm_{F/\Q}(v)^{-1}$ and $||\det(h)|| = \Nm_{F/\Q}(v)^{-2}$, it follows from the first equation that
$$\bq_{c_t'}(g_\infty f) = \Nm_{F/\Q}(v) \!\!\!\!\!\! \!\!\!\!\!\! 
\sum_{m \in ((\gd\gn)^{-1}vJ)_+ \cup \{0\}}\!\!\!\!\!\! \!\!\!\!\!\!  r_m^{\varpi_v^{-1}t}(S_v f)q^m,$$
and from the second that
$$\bq_{c_t'}(g_\iota f) = \Nm_{F/\Q}(v)^{-1} \!\!\!\!\!\! \!\!\!\!\!\! 
\sum_{m \in ((\gd\gn v)^{-1}J)_+ \cup \{0\}}\!\!\!\!\!\! \!\!\!\!\!\! \zeta_{N\ell}^{-\epsilon_\iota(N\ell m)} r_m^{\varpi t}(f)q^m,$$
where $\ell$ is the rational prime in $v$.  Since $m \mapsto 
\zeta_{N\ell}^{-\epsilon_\iota(N\ell m)}$ runs through the distinct characters $(\gd\gn v)^{-1}J/(\gd\gn)^{-1}J \to \CO^\times$ as $\epsilon_\iota$ runs through the representatives of $\gn J^{-1}/v\gn J^{-1}$, it follows that
$$\sum_{\iota \in \CO_F/v} \!\!\! \zeta_{N\ell}^{-\epsilon_\iota(N\ell m)} = \left\{\begin{array}{cl}
 \Nm(v), & \mbox{if $m \in (\gd\gn)^{-1}J$;}\\
   0, & \mbox{otherwise.}\end{array}\right.$$
Summing over $\iota \in \PP^1(\CO_F/v)$ thus yields the desired formula. 

If $U = U_1(\gn)$, then part (1) follows from the case of $U = U(\gn)$.  Alternatively one can use the same argument as above with slight modifications to yield both (1) and (2).  Indeed the only changes needed are that the term with $\iota = \infty$ does not appear if $v|\gn$, we take $U' = U \cap U(v)$ instead of $U(\gn v)$, we choose $\epsilon_\iota \in J^{-1}/vJ^{-1}$ for $\iota \in \CO_F/v$, and the factor of $\gn$ disappears from the remaining expressions.
\end{proof}

\begin{remark} Note that by (\ref{eqn:rmt}), the formulas in Proposition~\ref{prop:Tvonq} agree with the ones in Propositions~9.5.1 and~9.6.1 of \cite{DS1}, where it is assumed $p$ is unramified in $F$ and the ideal denoted there by $J$ is our $\gd J_t^{-1}$.  The difference by a factor of $\gd$ in the description of $D_{\vec{k},\vec{m},R}$ and its counterpart in \cite{DS1} arises from our use here of the more natural definition  in \cite{theta} of the line bundles $\widetilde{\CA}_{\vec{k},\vec{m}}$ in terms of determinants of certain rank two subquotients of the de Rham cohomology of the universal abelian variety.
\end{remark}

\subsection{The operator $T_\gp$}
We now consider Hecke operators at primes $\gp|p$.  Recall that $T_{\gp}$ is defined on 
$M_{\vec{k},\vec{m}}(U;R)$ in \cite[\S5.4]{KS} under the assumption\footnote{unnecessary if $R$ is a $K$-algebra} that
\begin{equation}\label{eqn:inequality}
\sum_{\theta\in \Theta_\gp} 
\min\{m_\theta,m_\theta+k_\theta-1\} \ge 0.
\end{equation}
More precisely, if (\ref{eqn:inequality}) holds, then for any sufficiently small $U$ prime to $p$
and Noetherian $\CO$-algebra $R$ such that $\chi_{\vec{k}+2\vec{m},R} = 1$ on $\CO_F^\times \cap U$, the operator 
$T_\gp$ on $M_{\vec{k},\vec{m}}(U;R)$ is defined as the composite 
\begin{equation}\label{eqn:Tpdef}
\begin{array}{rcl} H^0(Y_R,\CA_{\vec{k},\vec{m},R}) &\longrightarrow&
H^0(Y_0(\gp)_R,\pi_2^*\CA_{\vec{k}-\vec{1},\vec{m},R}
   \otimes \pi_2^*\omega_R) \\
&\longrightarrow&
H^0(Y_0(\gp)_R,\pi_1^*\CA_{\vec{k}-\vec{1},\vec{m},R}
   \otimes \pi_2^*\omega_R) \\
&\stackrel{\sim}{\longrightarrow}&
H^0(Y_R,\CA_{\vec{k}-\vec{1},\vec{m},R} \otimes \pi_{1,*}\pi_2^*\omega_R)\\ 
&\longrightarrow&
H^0(Y_R,\CA_{\vec{k},\vec{m},R}),\end{array}
\end{equation}
where the first morphism is pull-back by $\pi_2$, the second is induced by
the universal isogeny over $\widetilde{Y}_0(\gp)$, the third is the projection formula, 
and the last is induced by the saving trace.
Furthermore the operator $T_\gp$ commutes with the action of $\GL_2(\A_{F,\f}^{(p)})$ in the obvious sense, so it defines a $\GL_2(\A_{F,\f}^{(p)})$-equivariant endomorphism of $M_{\vec{k},\vec{m},R}$.  We may therefore define the operator $T_\gp$ on $M_{\vec{k},\vec{m}}(U;R)$ for arbitrary $U$ (prime to $p$) by taking $U$-invariants.
\begin{theorem} \label{thm:Tp} Suppose that $\gp$ is a prime of $\CO_F$ over $p$ and that $\vec{k},\vec{m}\in \Z^\Theta$ satisfy (\ref{eqn:inequality}).  Then the operators $T_\gp$ on $M_{\vec{k},\vec{m}}(U;R)$ defined above induce a $\GL_2(\A_{F,\f}^{(p)})$-equivariant endomorphism of $M_{\vec{k},\vec{m},R}$
which preserves $S_{\vec{k},\vec{m},R}$, is compatible with base-changes $R \to R'$, and coincides with the classical Hecke operator $T_\gp$ if $R$ is a $K$-algebra (in which case $k_\theta + 2m_\theta$ is independent of $\theta$).
\end{theorem}
\begin{proof} The assertions are all immediate from \cite[Thm.~5.4.1]{KS}, except for the one concerning $S_{\vec{k},\vec{m},R}$.  So it remains to prove that $T_\gp$ preserves $S_{\vec{k},\vec{m}}(U;R)$ for sufficiently small $U$, which we may furthermore assume is of the form $U=U(N)$.

For each cusp $c \in C_0(\gp)$, consider the effect on $q$-expansions of the maps on completions
\begin{equation}\label{eqn:QtoQ} \begin{array}{rcl}
Q_2 := (i_*\CA_{\vec{k},\vec{m},R})_{Z_R^{(2)}}^\wedge &
  \longrightarrow & 
(j_*(\pi_2^*\CA_{\vec{k}-\vec{1},\vec{m},R} \otimes \pi_2^*\omega_R))_{Z_R}^\wedge \\
& \longrightarrow & 
(j_*(\pi_1^*\CA_{\vec{k}-\vec{1},\vec{m},R} \otimes \pi_2^*\omega_R))_{Z_R}^\wedge \\
& {\longrightarrow} & 
(i_*(\CA_{\vec{k}-\vec{1},\vec{m},R} \otimes \pi_{1,*}\pi_2^*\omega_R))_{Z_R^{(1)}}^\wedge \\
& \longrightarrow & 
(i_*\CA_{\vec{k},\vec{m},R})_{Z_R^{(1)}}^\wedge =: Q_1
\end{array}
\end{equation}
induced by the morphisms in (\ref{eqn:Tpdef}), where 
$i:Y \hookrightarrow Y^{\min}$, $j:Y_0(\gp) \hookrightarrow Y_0(\gp)^{\min}$, and $Z$ (resp.~$Z^{(1)}$, $Z^{(2)}$) is the component of the complement corresponding to~$c$ (resp.~$\pi_1^\infty(c)$, $\pi_2^\infty(c)$).
It suffices to prove that if the constant term of the $q$-expansion associated to an element of $Q_2$ vanishes, so does that of its image in $Q_1$.
To that end we may replace $Y$ by $\widetilde{Y}$, $Y_0(\gp)$ by $\widetilde{Y}_0(\gp)$, $c$ by any cusp $\widetilde{c} \in \widetilde{C}_0(\gp)$ lying over it, etc. 
Furthermore, by the Koecher Principle (for $\widetilde{\CA}_{\vec{k},\vec{m},R}$ over $\widetilde{Y}_R$), we may replace
minimal compactifications by toroidal and consider instead the morphisms on completions associated to the composite
$$\begin{array}{rcl} H^0(\widetilde{Y}^\tor_R,\widetilde{\CA}^{\tor}_{\vec{k},\vec{m},R}) &\longrightarrow&
H^0(\widetilde{Y}_0(\gp)^\tor_R,\widetilde{\CA}^{(2),\tor}_{\vec{k}-\vec{1},\vec{m},R}
   \otimes \widetilde{\pi}_2^*\widetilde{\omega}_R) \\
&\longrightarrow&
H^0(\widetilde{Y}_0(\gp)^\tor_R,\widetilde{\CA}^{(1),\tor}_{\vec{k}-\vec{1},\vec{m},R}
   \otimes \widetilde{\pi}_2^*\widetilde{\omega}_R) \\
&\stackrel{\sim}{\longrightarrow}&
H^0(\widetilde{Y}^\tor_R,\widetilde{\CA}^\tor_{\vec{k}-\vec{1},\vec{m},R} \otimes \widetilde{\pi}_{1,*}\widetilde{\pi}_2^*\widetilde{\omega}_R)\\ 
&\longrightarrow&H^0(\widetilde{Y}^\tor_R,\widetilde{\CA}^{\tor}_{\vec{k},\vec{m},R}).\end{array}$$
The desired morphism can therefore be realized as the composite 
of the restrictions to $(U\cap \CO_F^\times)^2$-invariants of the maps
\begin{equation}\label{eqn:DtoD} \begin{array}{rcl} D_{\vec{k},\vec{m}}^{(2)}\otimes_{\CO}\Gamma(\widehat{S}_2,\CO_{\widehat{S}_2})
&\longrightarrow&
D_{\vec{k},\vec{m}}^{(2)}\otimes_{\CO}\Gamma(\widehat{S},\CO_{\widehat{S}})\\ & =  &
D_{\vec{k}-\vec{1},\vec{m}}^{(2)}\otimes_{\CO}
(\bigwedge^d(I_2)^{-1} \otimes\Gamma(\widehat{S},\CO_{\widehat{S}})) \\
&\longrightarrow&
D_{\vec{k}-\vec{1},\vec{m}}^{(1)}\otimes_{\CO}
(\bigwedge^d(I_2)^{-1} \otimes\Gamma(\widehat{S},\CO_{\widehat{S}})) \\
&{\longrightarrow}&
D_{\vec{k}-\vec{1},\vec{m}}^{(1)}\otimes_{\CO}
(\bigwedge^d(I_1)^{-1} \otimes\Gamma(\widehat{S}_1,\CO_{\widehat{S}_1})) \\ 
&= & D_{\vec{k},\vec{m}}^{(1)}\otimes_{\CO}\Gamma(\widehat{S}_1,\CO_{\widehat{S}_1}),
\end{array}
\end{equation}
where $\widehat{S}$ (resp.~$\widehat{S}_i$) is the formal scheme whose quotient by $(U\cap \CO_F^\times)^2$ defines the
completion of $\widetilde{Y}_0(\gp)_R^\tor$ (resp.~$\widetilde{Y}_R^\tor$) along the preimage of $\widetilde{Z}_R$ (resp.~$\widetilde{Z}_R^{(i)}$),
so that
$$\Gamma(\widehat{S},\CO_{\widehat{S}})
 = R[[q^m]]_{m \in N^{-1}M_+\cup\{0\}},\quad
 \Gamma(\widehat{S}_i,\CO_{\widehat{S}_i})
 = R[[q^m]]_{m \in N^{-1}M_{i,+}\cup\{0\}}$$
(with $M = \gd^{-1}I_1^{-1}J_2$, $M_i = \gd^{-1}I_i^{-1}J_i$),
and the morphisms $\Gamma(\widehat{S}_2,\CO_{\widehat{S}_2})
\to \Gamma(\widehat{S},\CO_{\widehat{S}})$ and 
$D_{\vec{k}-\vec{1},\vec{m}}^{(2)}
\to D_{\vec{k}-\vec{1},\vec{m}}^{(1)}$ are induced
(in the latter case thanks to (\ref{eqn:inequality}))
by the inclusions $I_2^{-1} \hookrightarrow I_1^{-1}$
and $J_2^{-1} \hookrightarrow J_1^{-1}$, and the last morphism
by the saving trace.  The desired conclusion is therefore
immediate from Proposition~\ref{prop:stq}.
\end{proof}

\subsection{The operator $S_\gp$} \label{ss:Sp}
As in the case of primes $v\nmid p$, the expression for its effect on $q$-expansions will involve the operator $S_\gp$, which may also be defined on $M_{\vec{k},\vec{m}}(\gn;R)$ for arbitrary $R$ under hypotheses on $(\vec{k},\vec{m})$.  More precisely, consider the automorphism $\widetilde{\rho}$ of $\widetilde{Y}$ defined in terms of the universal object $(A,\iota,\lambda,\eta,\CF^\bullet)$ by the data $(A',\iota',\lambda',\eta',\CF^{\prime,\bullet})$, where
\begin{itemize}
\item $A' = A \otimes_{\CO_F}\gp^{-1}$;
\item $\iota'$ is the $\CO_F$-action compatible with the isogeny $\sigma:A \to A'$;
\item $\lambda'$ is the quasi-polarization such that $\sigma^
{\vee}\circ\lambda'\circ \sigma = \lambda\circ\iota(\varpi_\gp^2)$;
\item $\eta'$ is the level $U$-structure compatible with $\sigma$;
\item $\CF^{\prime,\bullet}$ is the Pappas--Rapoport filtration corresponding to $\CF^\bullet\otimes_{\CO_F}\gp$ under the identification $s'_*\Omega^1_{A'/\widetilde{Y}} = (s_*\Omega^1_{A/Y})\otimes_{\CO_F}\gp$.
\end{itemize}
Note that since $\sigma$ induces isomorphisms
$$\widetilde{\rho}^*\omega_\theta \stackrel{\sim}{\longrightarrow} \theta(\varpi_\gp)\omega_\theta\quad\mbox{and}\quad
\widetilde{\rho}^*\delta_\theta \stackrel{\sim}{\longrightarrow} \theta(\varpi_\gp)^2\delta_\theta,$$
it induces a morphism $\widetilde{\rho}^*\widetilde{\CA}_{\vec{k},\vec{m}} \to \widetilde{\CA}_{\vec{k},\vec{m}}$ divisible by $\Nm(\gp^2)$ provided
\begin{equation}\label{eqn:inequality2} \sum_{\theta\in \Theta_\gp} (k_\theta + 2 m_\theta) \ge 2e_{\gp}f_{\gp}.\end{equation}
Furthermore the automorphism $\widetilde{\rho}$ descends to $Y$, as does the above morphism of sheaves (divided by $\Nm(\gp^2)$) to $Y_R$, yielding a morphism ${\rho}^*{\CA}_{\vec{k},\vec{m},R} \to {\CA}_{\vec{k},\vec{m},R}$ (assuming the above inequality
and the usual hypothesis that $\chi_{\vec{k}+2\vec{m},R}$ is trivial on $\CO_F^\times \cap U$, and denoting the automorphism of $Y_R$ by $\rho$).  We then define the endomorphism $S_\gp$ of $M_{\vec{k},\vec{m}}(U;R)$ to be the composite
$$H^0(Y_R,\CA_{\vec{k},\vec{m},R}) \longrightarrow
H^0(Y_R,\rho^*\CA_{\vec{k},\vec{m},R}) \longrightarrow
H^0(Y_R,\CA_{\vec{k},\vec{m},R}).$$

The operator $S_\gp$ is independent of the choice of $\varpi_\gp$, but we need to introduce a renormalization in order to describe the effect of $T_\gp$ on $q$-expansions.  Recall that the image of the morphism $\widetilde{\rho}^*\widetilde{\CA}_{\vec{k},\vec{m}} \to \widetilde{\CA}_{\vec{k},\vec{m}}$ induced by $\sigma$ has image $\chi_{\vec{k}+2\vec{m}}(\varpi_\gp)\widetilde{\CA}_{\vec{k},\vec{m}}$, so its composite with multiplication by
$\Nm(\gp)^{-1}\chi_{\vec{1}-\vec{k}-2\vec{m}}(\varpi_\gp)$ defines an isomorphism 
$\widetilde{\rho}^*\widetilde{\CA}_{\vec{k},\vec{m}} \stackrel{\sim}{\to} \widetilde{\CA}_{\vec{k},\vec{m}}$
for any $\vec{k},\vec{m}\in \Z^\Sigma$.
We then let $S_{\varpi_\gp}$ denote the automorphism of $M_{\vec{k},\vec{m}}(U;R)$ defined in the same as way as $S_\gp$, but using the resulting isomorphism ${\rho}^*{\CA}_{\vec{k},\vec{m},R} \stackrel{\sim}{\to}{\CA}_{\vec{k},\vec{m},R}$.
We therefore have the relation
$$S_\gp = \Nm(\gp)^{-1}\chi_{\vec{k}+2\vec{m} - \vec{1}}(\varpi_\gp)S_{\varpi_\gp}$$
whenever (\ref{eqn:inequality2}) holds (so that $S_\gp$ is defined, and $\Nm(\gp)^{-1}\chi_{\vec{k}+2\vec{m} - \vec{1}}(\varpi_\gp) \in \CO$).
Note that $S_{\varpi_\gp}$ depends on $\varpi_\gp$; more precisely if $\varpi'_{\gp} = \alpha\varpi_\gp$ for some $\alpha \in \CO_{F,(p),+}^\times$, then 
\begin{equation} \label{eqn:Spip} S_{\varpi_\gp} = \chi_{\vec{k}+2\vec{m} - \vec{1}}(\alpha)S_{\varpi'_\gp}.\end{equation}
  Alternatively, we may write $S_{\varpi_\gp} = ||x||[Ug^{-1}U]$, where\footnote{Beware the slight inconsistency in notation: $\varpi_v\in F_v^\times \subset (\A_{F,\f}^{(p)})^\times$ for $v\nmid p$, but $\varpi_\gp \in F^\times$ is diagonally embedded in $\A_{F,\f}^\times$ for $\gp|p$.}
$x = \varpi^{(p)}_\gp \in (\A_{F,\f}^{(p)})^\times$
and $g = \smat{x}{0}{0}{x} \in \GL_2(\A_{F,\f}^{(p)})$,
so that $||x|| = \Nm(\gp)\Nm(\varpi_\gp)^{-1} \in \Z_{(p)}^\times$.

It is immediate from the latter description that the operators $S_{\varpi_\gp}$ (as levels $U$ prime to $p$ and primes $\gp|p$ vary) define commuting $\GL_2(\A_{F,\f}^{(p)})$-equivariant automorphisms of $M_{\vec{k},\vec{m},R}$  preserving $S_{\vec{k},\vec{m},R}$.  In particular the $S_{\varpi_\gp}$ induce commuting automorphisms of $M_{\vec{k},\vec{m}}(U;R)$ preserving $S_{\vec{k},\vec{m}}(U;R)$ for all open compact subgroups $U$ of $\GL_2(\A_{F,\f}^{(p)})$ containing $\GL_2(\CO_{F,p})$.  It follows that the same is true for $S_\gp$ assuming  (\ref{eqn:inequality2}) holds, but with ``automorphism'' replaced by ``endomorphism.''

\subsection{$T_\gp$ on $q$-expansions} \label{ss:Tponq}
We now determine the effect of $T_\gp$ on $q$-expansions at cusps at $\infty$.

\begin{proposition} \label{prop:Tponq} Suppose that $\gp$ is a prime of $\CO_F$ dividing $p$, $\gn$ is an ideal of $\CO_F$ prime to $p$, $U$ is an open compact subgroup of $\GL_2(\A_{F,\f}^{(p)})$ containing $U(\gn)$, and let $f \in M_{\vec{k},\vec{m}}(U;R)$.  Then
$$r_m^t(T_\gp f) = \chi_{\vec{m}}(\varpi_\gp)r_{\varpi_\gp m}^{x^{-1}t}(f) +  \chi_{\vec{k}+\vec{m}-\vec{1}}(\varpi_\gp)r_{\varpi_\gp^{-1}m}^{xt}(S_{\varpi_\gp}f)$$
for all $t \in (\A_{F,\f}^{(p)})^\times$ and $m \in (\gd^{-1}\gn^{-1}J_t)_+ \cup \{0\}$, 
where $x = \varpi_{\gp}^{(p)}$.
\end{proposition}
\begin{proof} First note that the inequality (\ref{eqn:inequality}) implies that
$\chi_{\vec{m}}(\varpi_\gp)$ and $\chi_{\vec{k}+\vec{m}-\vec{1}}(\varpi_\gp)$ take values in $\CO$, and that each term in the expression is independent of the choice of $\varpi_\gp$
by (\ref{eqn:rmt}) and (\ref{eqn:Spip}). 

To prove the formula, we can once again assume $U = U(N)$ for some (sufficiently large) $N$ prime to $p$.  Recall that each cusp $c_t \in C$ has two elements in its preimage in $C_0(\gP)$ under $\pi_1^{\infty}$;  we let $c_t^{\et}$ (resp.~$c_t^{\mu}$) denote the one defined by the triple $(\underline{H},\underline{H}',\iota)$ with $H' = \gp^{-1}J_t \times \CO_F$ (resp.~$J_t \times \gp^{-1}$) and $\iota$ the inclusion. We then proceed as in the proof of Theorem~\ref{thm:Tp} to compute the effect on $q$-expansions of the morphism (\ref{eqn:QtoQ}) for $c = c_t^{\et}$ and $c_t^{\mu}$.  

In the first case, we have $\pi_2^{\infty}(c_t^{\et}) = c_{x^{-1}t}$, and in the notation of the proof of Theorem~\ref{thm:Tp}, we have $I_1 = I_2 = \CO_F$ and 
$$J_1  = J_t \hookrightarrow \gp^{-1}J_t \stackrel{\sim}{\longrightarrow} J_{x^{-1}t} = J_2$$
is multiplication by $\varpi_\gp$, as is
$M_1 = \gd^{-1}J_t \hookrightarrow \gd^{-1}J_{x^{-1}t} = M = M_2$.
It therefore follows from
Proposition~\ref{prop:stq} that the resulting composite in (\ref{eqn:DtoD}) is induced by
\begin{itemize}
\item the isomorphism $\Gamma(\widehat{S}_2,\CO_{\widehat{S}_2})
 \stackrel{\sim}{\longrightarrow} \Gamma(\widehat{S},\CO_{\widehat{S}})$ corresponding to the identity on
 $R[[q^m]]_{N^{-1}M_+\cup\{0\}}$,
\item $\chi_{\vec{m}}(\varpi_\gp): D_{\vec{k}-\vec{1},\vec{m}}^{(2)} = D_{\vec{k}-\vec{1},\vec{m},x^{-1}t} \to D_{\vec{k}-\vec{1},\vec{m},t}= D_{\vec{k}-\vec{1},\vec{m}}^{(1)} $,
\item the identity $\wedge^d(I_2)^{-1} = \wedge^d(I_1)^{-1}$,
\item and the map $\Gamma(\widehat{S},\CO_{\widehat{S}})
 \longrightarrow \Gamma(\widehat{S}_1,\CO_{\widehat{S}_1})$ corresponding to 
$$\begin{array}{ccc}
R[[q^m]]_{m \in N^{-1}M_{+} \cup\{0\} } & \to &
R[[q^m]]_{m \in N^{-1}M_{1,+} \cup\{0\} } \\ && \\
\sum r_mq^m  & \mapsto &
\sum r_{\varpi_\gp m}q^m.\end{array}$$
\end{itemize} 

We now proceed similarly for $c_t^\mu$, except that 
$\pi_2^{\infty}(c_t^\mu)$ is not necessarily a cusp at $\infty$
in the sense of \S\ref{ss:cuspsinf}.  Instead we have  
$\pi_2^{\infty}(c_t^\mu) = \rho^\infty(c_{xt})$, where 
$\rho$ is the automorphism of $Y$ defined in \S\ref{ss:Sp} and
$\rho^{\infty} = \cdot \otimes_{\CO_F} \gp^{-1}$ is the automorphism of $C$ extending it to $Y^{\min}$.  We therefore instead consider the effect on $q$-expansions of the composite of the maps on completions defined by $S_{\varpi_\gp}^{-1}$ and (\ref{eqn:QtoQ}).

In this case we have $J_1 = J_2 = J_t$ and
$$ I_1 = \CO_F \hookrightarrow \gp^{-1} \stackrel{\sim}{\longrightarrow} \varpi_\gp \gp^{-1} = I_2$$
is multiplication by $\varpi_\gp$, as is 
$M_2 = \gd^{-1} \varpi_\gp^{-1}\gp J_t =  \gd^{-1}J_{xt}
 \hookrightarrow \gd^{-1}J_t = M= M_1$.  
The composite in (\ref{eqn:DtoD}) is now induced by
\begin{itemize}
\item the map $\Gamma(\widehat{S}_2,\CO_{\widehat{S}_2})
 \longrightarrow \Gamma(\widehat{S},\CO_{\widehat{S}})$ corresponding to 
$$\begin{array}{ccc}
R[[q^m]]_{m \in N^{-1}M_{2,+} \cup\{0\} } & \to &
R[[q^m]]_{m \in N^{-1}M_{+} \cup\{0\} } \\ && \\
\sum r_mq^m  & \mapsto &
\sum r_{m}q^{\varpi_\gp m}.\end{array}$$
\item $\chi_{\vec{k}+\vec{m}-\vec{1}}(\varpi_\gp): D_{\vec{k}-\vec{1},\vec{m}}^{(2)} \to D_{\vec{k}-\vec{1},\vec{m},t}= D_{\vec{k}-\vec{1},\vec{m}}^{(1)} $,
\item the isomorphism $\wedge^d(I_2)^{-1} \stackrel{\sim}{\to} \wedge^d(I_1)^{-1}$ defined by multiplication by 
$\|x\|^{-1} = \Nm(\gp)^{-1}\Nm(\varpi_\gp)$,
\item and the isomorphism $\Gamma(\widehat{S}_2,\CO_{\widehat{S}_2})
 \stackrel{\sim}{\longrightarrow} \Gamma(\widehat{S},\CO_{\widehat{S}})$ corresponding to the identity on
 $R[[q^m]]_{N^{-1}M_+\cup\{0\}}$.
\end{itemize}

Letting $I_3 = \CO_F$, $J_3 = J_{xt}$, etc. denote the data associated to the cusp $c_{xt}$, it follows from Proposition~\ref{prop:koecher}(3) that the effect of 
$S^{-1}_{\varpi_\gp} = \| x \|^{-1}[U \smat{x}{0}{0}{x} U]$ on $q$-expansions is given by the canonical isomorphism 
$D_{\vec{k},\vec{m}}^{(3)} = D_{\vec{k},\vec{m},xt} \stackrel{\sim}{\to} D_{\vec{k},\vec{m}}^{(2)}$, the map on power series induced by the identification $M_3 = M_2$, and multiplication by the
normalizing factor
of $ \| x \| =  \| x \|^{-1} \| x \|^2$.  

The proposition now follows from the fact that $\bq_{c_t}(T_\gp f)$ is the sum of the images of $\bq_{\pi_2^{\infty}(c^{\et}_t)}(f)$ and $\bq_{\pi_2^{\infty}(c^{\mu}_t)}(f) = \bq_{\pi_2^{\infty}(c^{\mu}_t)}(S_{\varpi_\gp}^{-1}(S_{\varpi_\gp} f))$ under the maps in (\ref{eqn:DtoD}).
\end{proof}

\begin{corollary} \label{cor:Tp} The operators $T_\gp$ commute (for varying $\gp|p$ such that (\ref{eqn:inequality}) holds).
\end{corollary}
\begin{proof} Suppose that $\gp$ and $\gq$ are primes dividing $p$.  By the $q$-expansion Principle, it suffices to prove that $T_\gp(T_\gq f)$ and $T_\gq(T_\gp f)$ have the same $q$-expansions at all cusps $c_t$ at $\infty$.  Note also that it follows from Theorem~\ref{thm:Tp} and the description of $S_{\varpi_\gp}$ as $\| x\|[Ug^{-1}U]$ that it commutes with $T_\gq$ and $S_{\varpi_\gq}$.  Applying Proposition~\ref{prop:Tponq} for both $\gp$ and $\gq$ therefore gives 
$$\begin{array}{l}
r_m^t(T_\gp(T_\gq f)) =  
\chi_{\vec{m}}(\varpi_\gp\varpi_\gq)r_{\varpi_\gp\varpi_\gq m}^{x^{-1}y^{-1}t}(f) +  \chi_{\vec{m}}(\varpi_\gp)\chi_{\vec{k}+\vec{m}-\vec{1}}(\varpi_\gq)r_{\varpi_\gp\varpi_\gq^{-1} m}^{x^{-1}yt}(S_{\varpi_\gq}f)\\
\quad + \chi_{\vec{k}+\vec{m}-\vec{1}}(\varpi_\gp)\chi_{\vec{m}}(\varpi_\gq)r_{\varpi_\gp^{-1}\varpi_\gq m}^{xy^{-1}t}(S_{\varpi_\gp}f)
+ \chi_{\vec{k}+\vec{m}-\vec{1}}(\varpi_\gp\varpi_\gq)r_{\varpi_\gp^{-1}\varpi_\gq^{-1} m}^{xyt}(S_{\varpi_\gp}S_{\varpi_\gq}f),
\end{array}
$$
where $y = \varpi_\gq^{(p)}$.  Interchanging the roles of $\gp$ and $\gq$ gives the same expression.
\end{proof}

Finally we remark that Proposition~\ref{prop:Tponq} gives another proof that $T_\gp$ coincides with the classical Hecke operator $T_\gp$ if $R$ is a $K$-algebra.

\section{Corrigenda to \cite{theta}} \label{sec:corr}
We list here several minor corrections to Sections~7.1 and~7.2 of \cite{theta}:

\begin{itemize}
\item p.33, $\ell.$-3:\, $B_1(\CO_{F,(p)})_+$ should be $B_1(\CO_{F,(p)})$.
\item p.34, $\ell.$2:\, $B_1(\CO_{F,(p)})gU$ should be $B_1(\CO_{F,(p)})gU^p$.
\item p.34, $\ell.$4:\, The description of the bijection assumes $U \subset \GL_2(\widehat{\CO}_F)$.
\item p.37, $\ell.$-3:\, The reference to [28, Thm.~2.5] should be to Theorem~4.4.10 of \cite{KWL:spl} (according to the reference numbering of this paper).
\item p.38, $\ell.$1:\, Assume throughout that $R$ is Noetherian.
%\item p.41, $\ell.$-12:\, $\chi_{\mathbf{l},R}(\nu)$ should be $\chi_{\mathbf{l},R}(\nu)^{-1}$
\item p.41, $\ell.$-6:\, %$\chi_{\mathbf{l},R}(\nu)$ should be $\chi_{\mathbf{l},R}(\nu)^{-1}$, and 
$\CO_{F,+}$ should be $\CO_{F,+}^\times$.

\end{itemize}

\section*{Acknowledgements} This work originated as an appendix to the forthcoming paper with Shu Sasaki~\cite{DS2}, but took on a life of its own. %was removed due to inflammation!
We are grateful to him, and also to Payman Kassaei, for numerous helpful conversations in relation to the topic.
The author also thanks Mladen Dimitrov for encouraging him to pursue some of the questions considered here, 
Najmuddin Fakhruddin for calling his attention to the notion of rational singularities (and hence the reference~\cite{danilov}),
and Kai-Wen Lan for quickly rectifying a minor oversight in \cite{KWL:spl} that came to light.  Finally we are grateful to the referee for calling attention to several points in an earlier version of this paper requiring clarification or minor correction.

\bibliographystyle{amsplain} 

\bibliography{MRrefs_comp} 

\providecommand{\bysame}{\leavevmode\hbox to3em{\hrulefill}\thinspace}
\providecommand{\MR}{\relax\ifhmode\unskip\space\fi MR }
\providecommand{\MRhref}[2]{%
  \href{http://www.ams.org/mathscinet-getitem?mr=#1}{#2}
}
\providecommand{\href}[2]{#2}
\begin{thebibliography}{10}

\bibitem{chai}
Ching-Li Chai, \emph{Arithmetic minimal compactification of the
  {H}ilbert-{B}lumenthal moduli spaces}, Ann. of Math. (2) \textbf{131} (1990),
  no.~3, 541--554. \MR{1053489}

\bibitem{danilov}
V.~I. Danilov, \emph{The geometry of toric varieties}, Uspekhi Mat. Nauk
  \textbf{33} (1978), no.~2(200), 85--134, 247. \MR{495499}

\bibitem{DMPhD}
Mariagiulia De~Maria, \emph{Hilbert modular forms modulo $p$ of partial weight
  one and unramifiedness of {G}alois representations}, Thesis
  (Ph.D.)--University of Lille and University of Luxembourg, 2020.

\bibitem{DP}
Pierre Deligne and Georgios Pappas, \emph{Singularit\'{e}s des espaces de
  modules de {H}ilbert, en les caract\'{e}ristiques divisant le discriminant},
  Compositio Math. \textbf{90} (1994), no.~1, 59--79. \MR{1266495}

\bibitem{KS}
Fred Diamond, \emph{Kodaira--{S}pencer isomorphisms and degeneracy maps on
  {I}wahori-level {H}ilbert modular varieties: the saving trace}, preprint,
  available at https://arxiv.org/abs/2111.10160, 2021.

\bibitem{theta}
\bysame, \emph{Geometric weight-shifting operators on {H}ilbert modular forms
  in characteristic {$p$}}, J. Inst. Math. Jussieu \textbf{22} (2023), no.~4,
  1871--1930. \MR{4611205}

\bibitem{DKS}
Fred Diamond, Payman Kassaei, and Shu Sasaki, \emph{A {${\rm mod}\,p$}
  {J}acquet-{L}anglands relation and {S}erre filtration via the geometry of
  {H}ilbert modular varieties: splicing and dicing}, Ast\'{e}risque (2023),
  no.~439, 111. \MR{4592495}

\bibitem{DS2}
Fred Diamond and Shu Sasaki, \emph{A {S}erre weight conjecture for geometric
  {H}ilbert modular forms in characteristic $p$, {I}{I}}, in preparation.

\bibitem{DS1}
\bysame, \emph{A {S}erre weight conjecture for geometric {H}ilbert modular
  forms in characteristic {$p$}}, J. Eur. Math. Soc. (JEMS) \textbf{25} (2023),
  no.~9, 3453--3536. \MR{4634676}

\bibitem{Dim}
Mladen Dimitrov, \emph{Compactifications arithm\'etiques des vari\'et\'es de
  {H}ilbert et formes modulaires de {H}ilbert pour {$\Gamma_1(\mathfrak
  c,\mathfrak n)$}}, Geometric aspects of {D}work theory. {V}ol. {I}, {II},
  Walter de Gruyter, Berlin, 2004, pp.~527--554. \MR{2099078}

\bibitem{DW}
Mladen Dimitrov and Gabor Wiese, \emph{Unramifiedness of {G}alois
  representations attached to {H}ilbert modular forms {${\rm mod}\,p$} of
  weight {$1$}}, J. Inst. Math. Jussieu \textbf{19} (2020), no.~2, 281--306.
  \MR{4079146}

\bibitem{ERX}
Matthew Emerton, Davide Reduzzi, and Liang Xiao, \emph{Unramifiedness of
  {G}alois representations arising from {H}ilbert modular surfaces}, Forum
  Math. Sigma \textbf{5} (2017), e29. \MR{3725733}

\bibitem{FP}
Najmuddin Fakhruddin and Vincent Pilloni, \emph{Hecke operators and the
  coherent cohomology of {S}himura varieties}, J. Inst. Math. Jussieu
  \textbf{22} (2023), no.~1, 1--69. \MR{4556929}

\bibitem{KWL:PhD}
Kai-Wen Lan, \emph{Arithmetic compactifications of {PEL}-type {S}himura
  varieties}, London Mathematical Society Monographs Series, vol.~36, Princeton
  University Press, Princeton, NJ, 2013. \MR{3186092}

\bibitem{KWL:IMRN}
\bysame, \emph{Integral models of toroidal compactifications with projective
  cone decompositions}, Int. Math. Res. Not. IMRN (2017), no.~11, 3237--3280.
  \MR{3693649}

\bibitem{KWL:spl}
\bysame, \emph{Compactifications of splitting models of {PEL}-type {S}himura
  varieties}, Trans. Amer. Math. Soc. \textbf{370} (2018), no.~4, 2463--2515.
  \MR{3748574}

\bibitem{PR}
G.~Pappas and M.~Rapoport, \emph{Local models in the ramified case. {II}.
  {S}plitting models}, Duke Math. J. \textbf{127} (2005), no.~2, 193--250.
  \MR{2130412}

\bibitem{pappas}
Georgios Pappas, \emph{Arithmetic models for {H}ilbert modular varieties},
  Compositio Math. \textbf{98} (1995), no.~1, 43--76. \MR{1353285}

\bibitem{rap}
M.~Rapoport, \emph{Compactifications de l'espace de modules de
  {H}ilbert-{B}lumenthal}, Compositio Math. \textbf{36} (1978), no.~3,
  255--335. \MR{515050}

\end{thebibliography}

\end{document}